\documentclass[12pt]{article}
\usepackage{amsmath,amssymb,amsfonts,amsthm}
\usepackage{eucal} 
\usepackage{graphicx}

\newcounter{ecount}
\newenvironment{ex}[1][Example]{\begin{trivlist}
                                 \item[\hskip \labelsep {\bfseries #1}]}{\end{trivlist}}
\newenvironment{remark}[1][Remark]{\begin{trivlist}
                                 \item[\hskip \labelsep {\bfseries #1}]}{\end{trivlist}}
\newenvironment{strategy}[1][Strategy]{\par\noindent\textbf{Strategy} }{}
\newcommand{\slz}{SL(2,\mathbb{Z})}
\newcommand{\cF}{\mathcal{F}}
\newcommand{\cH}{\mathcal{H}}
\newcommand{\C}{\mathbb{C}}
\newcommand{\R}{\mathbb{R}}

\newcommand{\Z}{\mathbb{Z}}
\newcommand{\cO}{\mathcal{O}}

\newcommand{\cM}{\mathcal{M}}

\newcommand{\cL}{\mathcal{L}}
\newcommand{\cI}{\mathcal{I}}
\newcommand{\cJ}{\mathcal{J}}

\newcommand{\cK}{\mathcal{K}}

\newcommand{\Ctn}{(\C^2)^{[n]}}
\newcommand{\bW}{W}
\newcommand{\bE}{E}

\newcommand{\Pp}{\mathbb{P}^1}
\newcommand{\Pt}{\mathbb{P}^2}
\newcommand{\Jt}{\tilde{J}}

\newcommand{\lang}{\left\langle}
\newcommand{\rang}{\right\rangle}

\newcommand{\lv}{\left |}

\newcommand{\Cx}{\mathbb{C}^\times}
\newcommand{\slc}{\widehat{sl_2}\C}
\newcommand{\fock}{\mbox{$\bigwedge$}^\infty V}
\newcommand{\focko}{\mbox{$\bigwedge$}^\infty_0 V}
\newcommand{\fockm}{\mbox{$\bigwedge$}^\infty_m V}

\DeclareMathOperator{\Td}{Td}
\DeclareMathOperator{\supp}{supp}

\DeclareMathOperator{\Ext}{Ext}
\DeclareMathOperator{\End}{End}
\DeclareMathOperator{\Tot}{Tot}
\DeclareMathOperator{\Lie}{Lie}

\DeclareMathOperator{\ch}{ch}

\DeclareMathOperator{\zee}{\mathfrak{z}}
\DeclareMathOperator{\Tr}{Tr}
\DeclareMathOperator{\hook}{h}
\DeclareMathOperator{\bl}{bl}
\DeclareMathOperator{\Sym}{Sym}
\DeclareMathOperator{\wt}{wt}

\newtheorem{Theorem}{Theorem}
\newtheorem{Lemma}{Lemma}

\newtheorem{Proposition}{Proposition}
\newtheorem{Claim}{Claim}

\author{Erik Carlsson}

\title{Vertex Operators and Moduli Spaces of Sheaves}

\begin{document}

\maketitle

\abstract{The \emph{Nekrasov partition function} in supersymmetric quantum gauge theory
is mathematically formulated as an 
equivariant integral over certain moduli spaces of sheaves on a complex surface.
In \emph{Seiberg-Witten Theory and Random Partitions}, Nekrasov and Okounkov
studied these integrals using the representation theory of ``vertex operators'' and
the infinite wedge representation. Many of these operators arise naturally
from correspondences on the moduli spaces, such as Nakajima's Heisenberg operators,
and Grojnowski's vertex operators.

In this paper, we build a new vertex operator out of the Chern class of a vector bundle
on a pair of moduli spaces. This operator has the advantage that it
connects to the partition function by definition. It also incorporates the 
canonical class of the surface, whereas many other studies assume
that the class vanishes. When the moduli space is the Hilbert scheme, we present
an explicit expression in the Nakajima operators, and the resulting combinatorial identities.

We then apply the vertex operator to the above integrals. 
In agreeable cases, the commutation properties of the
vertex operator result in \emph{modularity} properties of the partition function and 
related correlation functions.
We present examples in which the integrals are completely determined by their modularity,
and their first few values.}

%\dedication{}

\section{Introduction}

The original motivation for this paper is a collection of (equivariant) integrals on
the moduli space of framed torsion-free sheaves on $\Pt$. These functions are the 
\emph{Nekrasov partition functions}, and the related correlation functions,
defined below. In \cite{NO}, Okounkov and Nekrasov
showed that they may be studied using the representation theory of the infinite
wedge representation and vertex operators. The vertex operators  are auxilliary
in the sense that they are not directly related to the partition functions; they are
tools for performing the integration. 

The vertex operator is the central object in this paper. 
In fact, it is as important to us as the original integrals.
We present a geometric (in fact, $K$-theoretic) construction
of the vertex operator, with the original intent of generalizing the above methods to
other surfaces.  It's geometric definition is more
general, and is therefore a better logical starting point for this paper than
its representation-theoretic definition. Therefore, we begin by constructing the vertex operator,
and present it's applications to the partition and correlation functions second.

\subsection{Hilbert scheme structures}

Given a smooth quasi-projective surface $S$, let $S^{[n]} = Hilb^n(S)$ denote the Hilbert scheme
of $n$ points on $S$. Assume that if $S$ is not projective, it carries an action
of a complex torus $T$ with compact fixed loci, inducing such an action on
$S^{[n]}$. The point here is that integration make sense on $S^{[n]}$, either
in the usual sense for $S$ projective, or in the sense of \emph{localization} for
$S$ equivariant. Let
$T_mS^{[n]}$ denote the tangent bundle to $S^{[n]}$ together with an action of
$\Cx$ by scaling the fibers. We consider the generating functions
\[F(k_1,...,k_N;m,q) = \]
\begin{equation}
\sum_{n\geq 0} q^n
\int_{S^{[n]}} \ch_{k_1}(\cO/\cI)\cdots \ch_{k_N}(\cO/\cI) e(T_mS^{[n]}).
\label{correlations}
\end{equation}
Here $\cO/\cI$ is the tautological vector bundle on $S^{[n]}$ whose fiber over an
ideal sheaf $(I \subset \cO) \in S^{[n]}$ is $H^0(\cO/I)$. $T_mS^{[n]}$ is an
\emph{equivariant} bundle on $S^{[n]}$, whether or not $S$ is equivariant,
and $e(T_mS^{[n]})$ is the equivariant Euler class. It depends on a parameter
$m \in \Lie(\Cx) = \C$, hence the subscript in $T_m$.

On one hand, this sort of integral can be thought of as a \emph{correlation function}
corresponding to a the partition function in the Nekrasov theory, reviewed
in section \ref{neksection}. On the other
hand, it represents a collection of numerical invariants of $S^{[n]}$: 
if $S$ is projective, the class $\ch_{k_1}(\cO/\cI)\cdots \ch_{k_N}(\cO/\cI)$ may not lie in
the top-dimensional cohomology of the Hilbert scheme, hence integrating it returns 0.
However, the equivariant Euler class as a function of $m$ is
\[e(T_mS^{[n]}) = \sum_j m^j c_{2n-j}(TS^{[n]}),\]
so the Euler class term serves to fill out the remaining degrees necessary to
give a top-dimensional cohomology class with the suitable Chern class of
the tangent bundle. Given that $k_1,...,k_N$ are fixed, filling out the 
remaining terms with the tangent bundle is necessary for a
generating function with more than one nonzero term if $S$ does not have a torus action.

Combining the correlation functions in a generating functions is crucial.
On $S^{[n]}$ for fixed $n$, operations such as the cup product are
difficult to understand. However, when the Hilbert schemes ranging over all $n$ are considered
together, the cohomologies inherit richer structures. That is the
strategy we employ for studying the generating function $F$: each
coefficient $[q^n] F(q)$ is difficult to understand in a meaningful way.
The generating function, on the other hand, admits a 
construction by canonical operators on
\[\cH = \bigoplus_n \cH_n = \bigoplus_n H^*(S^{[n]},\C).\]
Ideally, commutation relations between such operators
lead to properties of $F$ as a function of the generating variable
$q$.

The first of these structures are the \emph{Nakajima operators}. For
each cohomology class $c\in H^*(S,\C)$, Nakajima defines an operator
\[\alpha_k(c) : \cH_n \rightarrow \cH_{n-k},\]
by constructing a family of canonical correspondences defined below.
They collectively satisfy the relation
\[[\alpha_k(c),\alpha_l(c')] = k\delta_{k,-l}\lang c, c' \rang_S.\]
In addition to the commutation relation above, they have the advantage
that $\cH$ is an irreducible representation of the Heisenberg algebra
they generate.

Independently, Grojnowski constructed a different set of correspondences
on the Moduli spaces of torsion-free sheaves \cite{Groj} on $S$. Let
$\cM_{c_1,\ch_2}$ be the space of those sheaves $F$ with rank 1, $c_1(F) = c_1$,
$\ch_2(F) = \ch_2$. Then $\cM_{0,n}$ coincides with the Hilbert scheme, and more generally,
we have an identification $\cM_{c_1,\ch_2} \cong S^{[n]}$ where $n = \lang c_1,c_1\rang/2-\ch_2$
(see Nakajima \cite{Nak2}, chapter 9).

In this formulation, the bi-graded vector space
\[V = \bigoplus_{c,n} H^*(\cM_{c,n}).\]
inherits the extra structure of a \emph{vertex algebra}. We do not
discuss vertex algebras here, but remark that this structure gives operators
\[Y(x,z) \in \End(V,V)\otimes\C((z))\]
for each cohomology class $x \in V$. $\C((z))$ is the ring 
of formal Laurent series in $z$. Each operator admits an expression in
the Nakajima operators, and isomorphisms $Q^m : V_{c,n} \rightarrow V_{c+m,n}$ using
the identification with the Hilbert scheme. Conversely, the Nakajima operators may
be reconstructed from the vertex operators, so in some sense
the two approaches are equivalent.

However, the connection between the vertex operator $Y(x,z)$ and our vertex operator
is still mysterious. On one hand, the input into our operator is the first
Chern class of a line bundle $\cL$ on $S$. If
$S=\C^2$ with the torus action (\ref{idealsetup}), and we specialize
$x=1 \in H^0(\cM_{c_1(\cL),c_1^2/2})$, the two operators actually agree up to the operator $Q$.
On the other hand for general surfaces (for instance, any surface with nontrivial
canonical bundle $\cK$), our operator differs from the vertex 
operators $Y(x,z)$ by a $c_1(\cK)$ term. This causes asymmetric commutation
relations with the Heisenberg operators $\alpha_k$ that are not present in
actual vertex algebras.

\subsection{A Prototypical situation}

As we pointed out, one desires that the correlation functions (\ref{correlations})
admit an expression in terms of some of the natural operators on the Hilbert scheme.
In an ideal situation, properties (commutation relations) of the operators lead to 
functional properties of $F$ in the generating variable $q$. 
We show in theorem \ref{quasiT}, that the setup with
\begin{equation}
\label{idealsetup}
 S = \C^2,\quad T = \Cx,\quad z\cdot(x,y) = (zx,z^{-1}y)
\end{equation}
is such a situation.

The calculation proceeds as follows. The operator $\cH_n\rightarrow \cH_n$
given by $x\mapsto x \cup \ch_n(\cO/\cI)$ admits is expressible in the
vertex operators mentioned above. Unfortunately, there is no such expression
for the operator $x\mapsto x \cup e(T_mS^{[n]})$. But, if $|\mu \rangle$
is the element corresponding $1$ in $\cH_\mu$, and
\[\lang x, y\rang = \int x\cup y,\]
the matrix element
\begin{equation}
\label{me}
\langle x \cup e(T_{\mu t}S^{[n]}), x \rangle = \int \left(x \cup e(T_{\mu t}S^{[n]})\right) \cup x 
\end{equation}
is the matrix element $\langle Q^{-\mu} Y(|\mu \rangle, 1) x, x \rangle$. 
As a result, the correlation function takes the alternative form
\[\sum_n q^n \int_{S^{[n]}} c \cup e(T_mS^{[n]}) = \]
\[\Tr q^d (c \cup \_) Q^{-\mu} Y(|\mu\rangle,1),\]
where $d$ multiplies $\cH_n$ by $n$.
The point is that the operator $Y(|\mu\rangle,1)$ has off-diagonal terms (indeed, terms
between different Hilbert schemes), though they do not affect the trace. However,
we must include them if we are to exploit the commutation relations of
$Y(|\mu\rangle,1)$. The restriction to $m=\mu t$ where $\mu \in \Z$ is harmless
since the correlation function is a polynomial in $m$ for fixed $n$.
Furthermore, the vertex operator as constructed in this paper is defined for
all $m\in \C$.

As promised, the expression of $F$ as a trace of canonical operators
on $\cH$ leads to a functional properties of the correlation functions.
We prove in theorem \ref{quasiT} that, as a function of $q$,
$F(k_1,...,k_N;m,q)$ is a \emph{quasimodular form} in of a certain weight
depending on $k_j$. As a result, one may determine an integral 
over $S^{[n]}$ by simply determining which quasimodular form represents
$F$, and then extracting the coefficient of $q^n$.
This type of calculation is the prototype for all 
others of interest to us. Essentially every
construction in this paper is designed with this in mind.

\subsection{Outline}

With the model scenario above as motivation, we outline the format of this
paper. Section \ref{coho} summarizes of all necessary background on
the Hilbert scheme. This includes the Nakajima operators in both the equivariant
and non-equivariant case. It also includes the relationship between the Nakajima
operators, and the \emph{fixed point basis} of equivariant cohomology. We
give this description when $S = \C^2$ with the
action of the 2-torus $T^2\subset GL(2)$.

Section \ref{vo} describes the correlation functions in detail, and motivates
a geometric definition of the vertex operator, $\bW$, the central object of this paper.
This definition has several advantages:
\begin{enumerate}
\item $\bW$ incorporates the canonical class of $S$, where other studies assume
it vanishes. This leads to an asymmetry which is not present in standard vertex operators.
\item The coefficient (\ref{me}) is automatically a matrix element of $\bW$
for any $S$. This follows directly from definition, not a calculation.
\item $m$ need not be integral.
\item $\bW$ is defined as a characteristic class of a vector bundle.
Therefore it is likely that theorem \ref{mainT} generalizes to an operator
on $K$-theory.
\end{enumerate}
Section \ref{vo} concludes with our main result, which presents $\bW$
as an expression $\Gamma$ in the Nakajima operators. This expression justifies
the reference to $\bW$ as a vertex operator. The proof of theorem \ref{mainT}
is the content of section \ref{proof}.

Section \ref{quasi} applies $\bW$ to the correlation functions for (\ref{idealsetup}).
We use the commutation relations among vertex operators to prove 
the quasimodularity result mentioned above (theorem \ref{quasiT}). We give
explicit combinatorial expressions for the correlation functions, and connect
them with quasimodular forms in simple examples.

Finally, section \ref{sheaves} applies the vertex operator to the moduli space
$\overline{\cM(2,n)}$ of framed torsion-free sheaves of rank $r=2$ on $\Pt$. We show that $\bW$
applies the (dual) partition function $Z^{\vee}$ in the Nekrasov theory.
We present a theorem describing $Z^{\vee}$ as a vector-valued function in
the generating variable $q$ satisfying a twisted modularity property (theorem \ref{vectorT}). 

We will see that the partition function is considerably more complicated in higher rank. In fact,
the proof of theorem \ref{vectorT} relies on the \emph{principle vertex operator construction}
for the affine Lie algebra $\slc$. We do not yet have a geometric explanation for the appearance of
$\slc$, but it seems likely that a generalization of our operator connects
with the $\slc$ action described by Licata \cite{Lic}, and the oscillator action described
by Baranovsky \cite{Ba}. We present examples
for certain values of the parameter $m$ for which $Z^\vee$ may be classified by regular modular
forms.

\section{Cohomology of the Hilbert Scheme}
\label{coho}
Let $S$ be a smooth quasi-projective surface, possibly together with
an action of a complex torus, $T=(\Cx)^d$.  $S^{[n]}$ is a
complex manifold of dimension $2n$ by Fogarty's theorem. Furthermore,
the action of $T$ on $S$ induces an action on $S^{[n]}$, 
which has isolated fixed points if $S$ does.
Thus we may consider the cohomology
groups $H_T^*(S^{[n]},\C)$. We always consider cohomologies over $\C$ in this paper,
even though much of what follows makes sense over $\mathbb{Q}$. 

It is well understood that it is best to study the cohomologies of $S^{[n]}$ for $n \geq 0$
simultaneously instead of fixing $n$. This is foreshadowed by
the fact that the Betti numbers are best expressed as a generating function, called G\"{o}ttsche's
formula \cite{Goe2}:
\[\sum_{n=0}^\infty \sum_{k=0}^{4n} b_k(S^{[n]})p^{k}q^n = 
\prod_{n=1}^\infty \prod_{k=0}^4 (1-(-1)^k p^{k+2n-2}q^n)^{(-1)^{k+1}b_k(S)}.\]

However, the real reason is that the vector space
\begin{equation}
\cH = \bigoplus_{n \geq 0} H_T^*(S^{[n]})
\end{equation}
inherits extra structure, the \emph{Nakajima operators}.

\subsection{Nakajima's operators}
Define a correspondence on Hilbert schemes via the \emph{incidence variety}
in $S^{[n]}\times S \times S^{[n+k]}$ \cite{Nak1},
\[Z^{n,n+k} = \{(I,x,J) | J \subset I, \supp(I/J) = x\}.\] 
It is shown that $[Z]$ has dimension $2n+k+1$, and possesses a well-defined fundamental class 
$[Z]$ in $H^{4n+2k-2}_T(S^{[n]}\times S \times S^{[n+k]})$.

Given a cohomology class $c \in H^*_T(S)$, the \emph{Nakajima operator}
\[\alpha_{-k}(c) : H^*_T(S^{[n]}) \rightarrow H_T^{*+2k+\deg(c)-2}(S^{[n+k]})\]
is defined by
\begin{equation}
\alpha_{-k}(c)(y) = p_{3!}((p_2^*(c) \cup p_1^*(y)) \cup [Z])
\end{equation}
where $p_{3!}$ denotes the Gysin homomorphism for the proper map $p_3 : Z \rightarrow S^{[n+k]}$.

For $k>0$, we also define 
\begin{equation} \label{alphadual}
\alpha_k(c) = (-1)^{k+1} \alpha_{-k}(c)^*,
\end{equation}
where the dual operator is with respect to the inner product on $\cH$ given by
$\lang x,y\rang = \int_{S^{[n]}} x \cup y$. If $S$ is compact, this takes
$\cH$ to itself. If $S$ is noncompact with a group action, it takes
$\cF = \cH \otimes_R F$ to itself, where $R = H_T^*(pt) = \C[\Lie(T)^*]$, and 
$F=\C(\Lie(T)^*)$, the fraction field of $R$.
\begin{remark} Notice that this definition depends
on the inner-product being defined, which is acceptable since we assumed that integration
makes sense on our surface. However, it is possible to define the Nakajima operator
for $k>0$ if this is not the case by taking $c$ in cohomology with compact support.
This would be necessary, for instance, if we considered $S=\C^2$ with no
group action.
\end{remark}

Nakajima's main result \cite{Nak1} is that
\begin{equation}
[\alpha_{k}(c),\alpha_{l}(c')] = m\delta_{k,-l} \lang c,c'\rang.
\end{equation}
Here the bracket is the \emph{supercommutator} with respect to the grading given
by $\deg(c)$. If $S$ has no odd cohomology, it reduces to a normal commutator.
One can also express this theorem by saying that $\cH$ is a module over the Heisenberg
Lie algebra with generators $\alpha_k(c)$. Furthermore, using G\"{o}ttsche's formula,
one can prove that this module is irreducible. Therefore, if $|0\rangle$ represents 
the element $1 \in H^0(S^{[0]}) \subset \cH$, we obtain a basis of $\cH$ by
successive applications of $\alpha_{-k}(c)$ to $|0\rangle$.

\subsection{Equivariant cohomology of the Hilbert scheme}
Suppose now that $S$ is equivariant with isolated fixed points, though not necessarily compact.
The action of $T$ on $S$ induces an action on the Hilbert scheme $S^{[n]}$, also having isolated
fixed points. In this case, we have another canonical basis for $\cH$, the
\emph{fixed-point basis} for equivariant cohomology. Since
$T$ acts on $S^{[n]}$ with isolated fixed points, localization gives an 
injective map of $R$-modules,
\[i^* : H_T^*(S^{[n]}) \rightarrow H_T^*(\bigsqcup_{s \in S^{[n]}_T} s) \cong
\bigoplus_{s\ fixed} H_T^*(s) \cong \bigoplus_{s\ fixed} R_s\]
where $i$ is the inclusion $\bigsqcup_{s \in S^{[n]}_T} \{s\} \hookrightarrow S^{[n]}$.
Tensoring over $R$ with $F$, this map becomes an isomorphism. Thus, to
describe a basis of $\cF = F \otimes_R \cH$, we need to determine the fixed points
of $S^{[n]}$.

We describe the fixed points in the setup 
\begin{equation}
S = \C^2,\quad T = \Cx \times \Cx\circlearrowleft S, \quad (z_1,z_2) \cdot (x,y) = (z_1x,z_2y).
\label{csetup}
\end{equation}
A fixed point of $\Ctn$ as a manifold is given by a closed point of $\Ctn$ as a 
scheme. Using the functor of points, these are in bijection with ideals of $I \subset \C[x,y]$
such that $\dim_\C \C[x,y]/I = n$, and $I$ is invariant under $T$. It is not
difficult to see that invariant ideals are just those that are generated by
monomials $x^ry^s$. Such ideals are indexed by partitions 
$\lambda = (\lambda_1,\lambda_2,...)$ of $n$, the correspondence being
\[\lambda \mapsto I_\lambda = (x^{\lambda_1},x^{\lambda_2}y,...,y^{\ell(\lambda)})\]
where $\ell(\lambda)$ is the smallest non-negative integer with $\lambda_\ell \neq 0$. 
A $T$-equivariant basis of $\C[x,y]/I_\lambda$ is $x^iy^j$ over all $(i,j)$ 
in the diagram of $\lambda$. The basis element of $\cF_S$ corresponding to $\lambda$
is $[I_\lambda] = i_{\lambda!}(1)\in \cH$, where $i_\lambda : pt \rightarrow S^{[n]}$
is the inclusion of the fixed point $I_\lambda$.

If $S$ is not $\C^2$ but has isolated fixed points, we can restrict to the compact
torus $S^1 \times \cdots \times S^1 \subset \Cx \times \cdots \times \Cx$, and find disjoint
equivariant neighborhoods $U_s$ of each fixed point $s\in S$. Let
$U_s^{[n]} \subset S^{[n]}$ consist of points with support
in $U_s$. This induces an injection
\[\coprod_{\sum_s n_s = n} \quad \prod_{s\ fixed} U_s^{[n_s]} \hookrightarrow S^{[n]}\]
by taking the union of subschemes of $S$ with disjoint support. By the
localization theorem \cite{AB}, this induces an isomorphism
\begin{equation}
\label{sdecomp}
\cF_S \rightarrow \bigotimes_{s\ fixed} \cF_{T_sS}.
\end{equation}
This map is easy to describe in the fixed point basis: fixed points of $S^{[n]}$ are
indexed by assigning a partition $\mu_s$ to each fixed point $s\in S$, describing the ideal
sheaf $I_{\mu_s}$ inserted at $s$. The image of this point is just $\bigotimes_s [I_{\mu_s}]$.

\subsection{Fixed points and Jack polynomials}

We now describe the relationship between the fixed-point basis and the Nakajima basis
in setup (\ref{csetup}). For more details we refer to Wang, Li and Qin \cite{LQW_Jack}, though 
we use the normalizations of Okounkov and Pandharipande \cite{OP}.

Let $\alpha_{-k} = \alpha_{-k}(1)$, and $\alpha_{\mu} = \prod_i \alpha_{-\mu_i}$
for $\mu$ a partition. The the elements $\alpha_{\mu}|0\rangle$ form 
a basis of $\cF$ over $F$. In fact,
\[\cF \cong F[\alpha_{-1},\alpha_{-2},...]\]
as $F$-vector spaces. The obvious ring structure is unrelated to
the cup product ring structure on $\cF$, but the Nakajima operators become
\begin{equation}
\alpha_{-k} \mapsto (f \mapsto \alpha_{-k}f), \quad
\alpha_{k} \mapsto \frac{k}{t_1t_2}\frac{\partial}{\partial \alpha_{-k}}
\end{equation}
Notice that, like the fixed-point basis,
the Nakajima basis is also indexed by partitions.

The standard way to express the fixed point basis in terms of the Nakajima basis
uses the isomorphism
\begin{equation}
\cF \rightarrow \Lambda^*\otimes \C(t_1,t_2), \quad \alpha_{\mu}|0\rangle \mapsto p_\mu
\end{equation}
where $\Lambda^*$ is the ring of symmetric polynomials in infinitely many variables,
and $p_\mu = \prod_k p_{\mu_k}$, and $p_k = \sum_i x_i^k$. The connection
between $\cF$ and symmetric polynomials is deeper than the
role it plays in this paper (see Haiman \cite{H});
for now, it merely provides a convenient combinatorial description of the 
change-of-basis matrix.

Under this isomorphism, the inner product
\begin{equation}
\label{iprod}
\lang x, y\rang = \int_{S^{[n]}} x \cup y
\end{equation}
takes the form
\begin{equation}
\label{jprod}
\lang p_\mu, p_\lambda \rang_{t_1,t_2} = \frac{(-1)^{|\mu|-\ell(\mu)}}{(t_1,t_2)^{\ell(\mu)}\zee(\mu)}\delta_{\mu,\lambda}
\end{equation}
where $\zee(\mu) = (\prod_i \mu_i)(\prod_k |\{i|\mu_i = k\}|!)$. The fixed
point elements transform as
\begin{equation}
\label{jack}
[I_\mu] \mapsto \Jt_\mu = t_2^{|\mu|}P_\mu^{(-t_2/t_1)}|_{p_n=t_1p_n}.
\end{equation}
Here $P^{(\theta)}$ is the \emph{integral form} of the Jack polynomial \cite{Mac}, and we have
normalized $\Jt_\mu$ so that $\langle\Jt_{\mu},\ p_1^{|\mu|}\rangle = |\mu|!$.
The reason for this normalization is that the element $\frac{1}{n!}\alpha_{-1}^n|0\rangle$
is $1 \in H^0_T(S^{[n]})$, and
\[\int i_{\mu!}(1) \cup 1 = 1.\]

If $S$ is not $\C^2$, the Nakajima operators can be easily understood in
the decomposition (\ref{sdecomp}) over $F$. Given positive integers $m_s,n_s$
with $\sum m_s = m$, $\sum n_s = n$, the restriction of $Z^{m,m+k}$ to
$\prod_s \left(U_s^{[m_s]} \times U_s^{[n_s]}\right)$ is
\[\coprod_{s,n_s-m_s = k} \Delta \times \cdots \times Z^{m_s+k,k} \times \cdots \times \Delta\]
as sets. It follows that $\alpha_{-k}(c)$ decomposes as
\begin{equation}
\label{alphadecomp}
\alpha_{-k}(c) = \sum_{s} \left(1\otimes \cdots \otimes 
\alpha_{-k}\left(c|_{U_s}\right)\otimes \cdots \otimes 1\right).
\end{equation}

\section{The Vertex Operator}
\label{vo}

In this section we consider integrals of Chern classes of tautological bundles
on $S^{[n]}$, where $S$ is a smooth quasi-projective surface, possibly equipped
with an action of a complex torus. If $S$ is not projective, neither is $S^{[n]}$, and
we must consider equivariant integrals with respect to a torus action. 
We present a strategy for studying these integrals that motivates the introduction
of the vertex operator, the centerpiece of this paper.

\subsection{Strategy}

Consider the integral
\[\int_{S^{[n]}} c \cup  e(T_mS^{[n]}).\]
Here $c$ is a cohomology class, $T_mX$ is the tangent bundle to $X$ taken as an equivariant
vector bundle with respect to an action of $\Cx$ acting
trivially on $X$, but scaling the fibers of $TX$. $e(T_mS^{[n]})$ is
the equivariant Euler class of this equivariant bundle, and the parameter
here is the number $m \in \Lie(\Cx)$, hence the subscript $m$.

If $c=1$ and $S$ is as in (\ref{idealsetup}), this is one of
Nekrasov's deformed partition functions, reviewed below. As pointed out in the introduction, it is also a 
tidy way to organize numerical invariants of $S^{[n]}$:
\[e(T_mS^{[n]}) = \sum_j m^j c_{2n-j}(TS^{[n]}).\]
So if $S$ is projective and $c$ has cohomological degree
$k$, the coefficient of $[m^k]$ fills out what is left to give a top degree class.
This means the integral makes sense, and yields a number associated to each $c$.

We summarize the strategy in $3$ steps:

\begin{strategy}{\ref{strat}}
\label{strat}
%\newline
\begin{enumerate}
 \item Make a wise choice of an element of $\cH \otimes \cH$ that pulls
back to $e(T_mS^{[n]})$ under $\Delta_{nn}^*$.
\item Using the standard inner product on cohomology 
$\lang a, b\rang = \int a \cup b$, this is the same as an operator
$\bW : \cH \rightarrow \cH$. It has the property that
\[\Tr\ (c \cup \_ ) \circ \bW = \int_{S^{[n]}} c \cup e(T_mS^{[n]})\]
A wise choice of $\bW$ admits an explicit expression in the Nakajima
operators.
\item Use the Nakajima commutation relations to study the trace. Alternatively,
compute the trace in the Nakajima basis.
\end{enumerate}
\end{strategy}

$\bW$ is the vertex operator.

\subsection{The Main theorem}

In fact, the vertex operator we construct is slightly more general. It
is well-known that the fiber of $TS^{[n]}$ at a point corresponding to 
an ideal $I$ is given by
\begin{equation}
T_IS^{[n]} \cong \chi_S(\cO,\cO)-\chi_S(I,I).
\end{equation}
Here 
\[\chi(F,G) = \sum_{i=0}^{2} (-1)^i \, \Ext^i(F,G),\]
and this is an isomorphism of virtual vector spaces. If $S$ is compact,
this is an alternating sum of finite-dimensional vector spaces, and we
expect the virtual dimension to be $2n = \dim(S^{[n]})$. If $S$ is
noncompact with a torus action fixing $I$, this represents a difference of infinite
vector spaces with finite-dimensional torus eigenspaces. Therefore the subtractions
are meaningful, and in fact we obtain the tangent space to $I$ as a virtual
representation of $T$, not just a virtual vector space.

Given a line bundle $\cL$ on $S$, we define an element
$\bE_{kl}(\cL) \in K^T(S^{[k]}\times S^{[l]})$ such that
\[\Delta_{nn}^*(\bE_{nn}(\cL))_{I} \cong T_I S^{[n]} \cong \chi_S(\cO,\cL)-\chi_S(I,I\otimes \cL).\]
If $\cL$ is a trivial bundle with an action of $\Cx$, we recover the class 
of $T_mS^{[n]}$.

Let $\bE_{kl}(\cL)$ be the virtual vector bundle on $S^{[k]} \times S^{[l]}$  with fiber
\begin{equation}
\label{E}
\bE\big|_{(I,J)} = \chi(\cO,\cL) - \chi(J,I\otimes\cL) 
\end{equation}
over a pair $(I,J)$ of ideal sheaves on $S$.
To be precise, let $\cI$, $\cJ$ be the canonical ideal sheaves on $S^{[k]}\times S$
and $S^{[l]}\times S$ respectively, given by the functor of points. Then define
\[E = (p_1\times p_3)_*(\cO^{\vee}\cdot \cO \cdot p_2^*(\cL)- 
\cJ^{\vee}\cdot \cI \cdot p_2^*(\cL)),\]
where $p_1,p_2,p_3$ are the projections onto $S^{[k]}$, $S$, $S^{[l]}$. 
We define $\bW(\cL)$ as the operator $\cF \rightarrow \cF$ (or $\cH \rightarrow \cH$ if $S$ is compact)
with matrix elements

\begin{equation}
  \label{defW}
 \left(\bW(\cL,z)\, p_1^*\eta, p_2^*\xi\right) = z^{l-k}\int_{S^{[k]} \times S^{[l]}} \eta \, \xi\, c_{k+l}(\bE) \,.
 \end{equation}
with $\eta \in H^*(S^{[m]})$,\ $\xi \in H^*(S^{[n]})$. If $S$ is compact, this is
an operator $\cH \rightarrow \C((z)) \otimes \cH$, where
$\C((z))$ is the ring of formal Laurent series in $z$. The coordinate $z$ is
necessary if we insist on defining $\cH$ as a direct sum, but later we
make use of its matrix elements when $z$ is assigned a numerical value.

As pointed out, if $\cL$ is the trivial bundle with an action of $u \in \Cx$ by scaling,
$\Delta_{nn}^* c_{2n}(E_{nn}) = e(\Delta^* E_{nn}) = e(T_mS^{[n]})$. Thus, $c_{k+l}(E_{kl})$ meets
the requirements of the cohomology class mentioned in the last subsection.
The first step in strategy (\ref{strat}) is then complete. The second is the content of our main theorem:

\begin{Theorem}
\label{mainT}
\begin{equation}
\label{mainE}
 \bW(\cL,z) = \Gamma(\cL,z) = \Gamma_-(\cL-\cK,z)\Gamma_+(\cL,z)
\end{equation}
where
\[\Gamma_{\pm}(\cL,z) = \exp\left(\sum_{n>0} \frac{z^{\mp n}}{n} \, \alpha_{\pm n}(\cL)\right)\]
%\[\Gamma_-(\cL,z) = \exp\left(\sum_n \frac{(-1)^{n-1}z^n}{n} \, \alpha_{-n}(\cL)\right),\]
%\[\Gamma_+(\cL,z) = \exp\left(\sum_n \frac{z^{-n}}{n} \, \alpha^*_{-n}(\cL)\right) \,.\]
\end{Theorem}
\begin{remark}
Serre duality on $S$ implies that
\[\chi(G,F\otimes \cL) = \left(\chi(F\otimes \cL,G)\otimes \cK\right)^*=
\chi(F,G\otimes\cK\otimes\cL^{-1})^*.\]
Since $c_k(\bE^*) = (-1)^kc_k(\bE)$, this shows that
\[\bW(\cL,z)^* = (-1)^d \bW(\cK-\cL,z^{-1}) (-1)^d,\]
which is immediately verified for $\Gamma(\cL,z)$, using (\ref{alphadual}).
\end{remark}
\begin{remark}
The theorem for $\bW(\cL,z)$ or $\bW(\cL,z)^*$ applied to $|0\rangle$ is already
known. In fact, it is a special case of a formula of Lehn: see \cite{L2}, theorem 4.6.
\end{remark}

\subsection{A Special case}

Before proving theorem, we demonstrate the statement in the setup (\ref{csetup}).
To do this, we need to calculate the matrix elements $(\bW(\cL)[I_\mu],[I_\lambda])$
in the fixed point basis. $\bW(\cL)$ is defined as a characteristic class,
so by the Atiyah-Bott localization formula and (\ref{E}), this is
\begin{equation}
\label{wfixed}
e(\bE|_{I_\mu,I_\lambda}) =
e(\chi(\cO,u\cdot\cO) - \chi(I_\lambda,u\cdot I_\mu))
\end{equation}
Where $e(V)$ is the product of the weights $\lambda \in \Lie(T)^*$ of
a representation $V$. To compute this coefficient, we need the character
of $\chi(\cO,\cO)-\chi(I_\lambda,I_\mu)$.
\begin{Lemma}\label{Lhooks} 
\begin{equation}
 \label{sum_hooks}
\ch\left(\bE\Big|_{(I_\lambda,I_\mu)}\right) = 
\sum_{\square\in \mu} z_1^{a_\lambda(\square)+1} \, z_2^{-l_\mu(\square)} +
\sum_{\square\in \lambda} z_1^{-a_\mu(\square)}\, z_2^{l_\lambda(\square)+1}, 
\end{equation}
where $a_\mu(\square), l_\mu(\square)$ denote the (possibly negative) arm and leg length in $\mu$.
\end{Lemma}
\begin{proof}
By localization and Grothendieck-Riemann-Roch for the equivariant map $p:\C^2\rightarrow pt$,
\[\ch \chi(I_\mu,I_\lambda) = \int_{\C^2} \ch(I_\lambda^\vee \cdot I_\mu) \Td(\C^2)=\]
\[\frac{\overline{i_0^*\ch(I_\lambda)}i_0^*\ch(I_\mu)}{t_1t_2}\frac{t_1t_2}{(1-z^{-1})(1-z_2^{-1})} =
\label{chichar}\frac{\overline{\ch(p_* I_\lambda)}\ch(p_* I_\mu)\ch(p_* \cO)}{\overline{\ch(p_* \cO)}\ch(p_* \cO)} = \]
\begin{equation}
\label{chicalc}
\frac{\ch(p_* I_\mu)\overline{\ch(p_* I_\lambda)}}{\overline{\ch(p_* \cO)}}
\end{equation}
where $i_0$ is the inclusion of $0\in\C^2$, check denotes the dual,
and $\overline{f(t_1,t_2)} = f(-t_1,-t_2)$. Let
$\mu^t$ denote the transposed partition. Substituting
\[
\ch (p_* I_\mu) = \sum_{j\ge 1} \frac{z_1^{1-j} z_2^{-\mu^t_j} }{1-z_2^{-1}}\,,
\quad 
\ch (p_* I_\lambda) = \sum_{i\ge 1} \frac{z_1^{-\lambda_i} \, z_2^{1-i}}{1-z_1^{-1}}\, 
\]
into (\ref{chicalc}) yields
\[\ch \chi(\cO,\cO)-\ch \chi(I_\lambda,I_\mu) = \sum_{i,j\ge 1} z_1^{\lambda_i-j+1} z_2^{i-\mu_j^t}  - 
\sum_{i,j\ge 1} z_1^{1-j} \, z_2^{i} \,.\]
Now observe that the terms for which the exponent of $z_2$ is $\le 0$ occur only in the first sum and correspond to the first 
sum in \eqref{sum_hooks}. The remaining terms may be determined using Serre duality which implies
\[\bE\Big|_{(I_\mu,I_\lambda)}  =  z_1 z_2 \left(  \bE\Big|_{(I_\lambda,I_\mu)}\right)^\vee \,.\]
%
%Note that the same argument with the roles of $z_1$ and $z_2$ interchanged yields the same formula 
%except the ranges of summation $\square\in \lambda$ and $\square\in \mu$ get interchanged. 
\end{proof}
The theorem now takes the purely combinatorial form
\[\lang \exp\left(\frac{m}{t_1t_2} \sum_n (-1)^{n+1} \frac{\partial}{\partial p_n}\right) \Jt_\mu,\ 
\exp\left(\frac{m+t_1+t_2}{t_1t_2} \sum_n \frac{\partial}{\partial p_n}\right) \Jt_\lambda\rang_{t_1,t_2} = \]
%\exp\left(\sum_n \frac{(-1)^{n-1}}{n} \, \alpha_{-n}(\cL)\right)
%\exp\left(\sum_n \frac{1}{n} \, \alpha^*_{-n}(\cL-\cK)\right) \,.
%
\begin{equation}
\label{combtheorem}
%K(\sum_{\square\in \mu} z_1^{-a_\lambda(\square)}\, z_2^{l_\mu(\square)+1} + 
%\sum_{\square\in \lambda} z_1^{a_\mu(\square)+1} \, z_2^{-l_\lambda(\square)}) = 
\prod_{\square \in \mu} (m+t_1a_\lambda(\square)+t_1-t_2l_\mu(\square))
\prod_{\square \in \lambda} (m-t_1a_\mu(\square)+t_2l_\lambda(\square)+t_2).
\end{equation}
The expression in (\ref{combtheorem}) is (\ref{wfixed}) using lemma \ref{Lhooks}.
\begin{ex}{$\mu=(2)$, $\lambda =(1,1).$}

We have
\[\Jt_{(2)} = t_1^2t_2^2p_1^2-t_1^2t_2p_2,\quad
\Jt_{(1,1)} = t_1^2t_2^2p_1^2-t_1t_2^2p_2.\]
Each term in the inner product in (\ref{combtheorem}) is
\[\exp\left(\frac{m}{t_1t_2} \sum_n (-1)^{n+1} \frac{\partial}{\partial p_n}\right)\Jt_{(2)} =\]
\[m(m-t_1)+(2mt_1t_2)p_1+(t_1^2t_2^2)p_1^2+(t_1^2t_2)p_2,\]
\[\exp\left(\frac{m+t_1+t_2}{t_1t_2} \sum_n \frac{\partial}{\partial p_n}\right)\Jt_{(1,1)} =\]
\[(m+t_1+2t_2)(m+t_1+t_2)+2t_2t_2(m+t_1+t_2)p_1+(t_1^2t_2^2)p_1^2+(-t_1t_2^2)p_2.\]
Taking the inner product (\ref{jprod}) gives
\[m(m+t_1)(m+t_1+t_2)(m-t_1+2t_2).\]
\end{ex}

As a warm up, we can now prove a special case of the theorem:
\begin{Lemma}
\label{striplemma}
Equation (\ref{combtheorem}) holds in the case $\mu = (1^k), \lambda = (l)$, $S$ and $T$
as in setup (\ref{csetup}), and $\cL = \C$ with a scaling action of $\Cx$.
\end{Lemma}
\begin{proof}
It is obvious from the definition of $Z^{n,n+k}$ that 
$\langle \alpha_{-n} [I_\mu], [I_\lambda] \rangle$ is zero unless
$\mu \subset \lambda$. Then
\[\langle \bW(\cL,z)[I_\mu],\ [I_\lambda]\rangle = \]
\[\sum_{\nu \subset \mu,\lambda} 
\frac{\langle [I_\mu],\ \Gamma_+(\cL,z)^*\alpha_\nu|0\rangle
\langle\Gamma_-(\cL-\cK,z) \alpha_\nu|0\rangle,\ [I_\lambda]\rangle}
{\langle \alpha_\nu|0\rangle,\alpha_\nu|0\rangle} =\]
\[\langle [I_\mu],\ \Gamma_+(m,z)^*|0\rangle
\langle\Gamma_-(m+t_1+t_2,z) |0\rangle,\ [I_\lambda]\rangle +\]
\[t_1t_2\lang [I_\mu],\ \Gamma_+(m,z)^*\alpha_{-1}|0\rang
\langle\Gamma_-(m+t_1+t_2,z) \alpha_{-1}|0\rangle,\ [I_\lambda]\rangle,\]
Since only $\emptyset$ and $[1]$ are contained in both $\mu$ and $\lambda$.
The cycle $Z^{n,n+1}$ is nonsingular for all $n$, so the operator $\alpha_{-1}$
is easy to understand in the fixed-point basis. We could either use an explicit
formula for the tangent bundle to $Z^{n,n+1}$ \cite{C}, or the Pieri rule for Jack polynomials \cite{Mac}.
The elements $\Gamma_-(\_,z)|0\rangle$, $\Gamma_+(\_,z)^*|0\rangle$ are Chern classes of bundles,
so they also have explicit expressions in the fixed-point basis:
\[\langle \alpha_{-1} [I_{\nu}],\ [I_{(1^{k+1})}]\rangle = 
(k+1)!t_2^k\prod_{j=0}^{k-1}(t_1-jt_2)\delta_{\nu,(1^k)},\]
\[\langle \alpha_{-1} [I_\nu],\ [I_{(l+1)}]\rangle = 
(l+1)!t_1^l\prod_{j=0}^{l-1}(t_2-jt_1)\delta_{\nu,(l)}\]
\[\langle [I_{(1^k)}],\ \Gamma_+(m,z)^*|0\rangle = z^{-k}\prod_{j=0}^{l-1}(m-jt_2)\]
\[\langle \Gamma_-(m+t_1+t_2,z)^*|0\rangle,\ [I_{(l)}]\rangle = z^{l}\prod_{j=0}^{k-1}(m+t_1+t_2+jt_2).\]
From localization,
\[\langle [I_{(1^k)}],\ [I_{(1^k)}]\rangle = k!t_2^k\prod_{j=0}^{k-1} (t_1-jt_2),\]
\[\langle [I_{(l)}],\ [I_{(l)}]\rangle = l!t_1^l\prod_{j=0}^{l-1} (t_2-jt_1).\]
Putting these together gives
\[\lang [I_{(1^k)}],\ \Gamma_+(m,z)^*\alpha_{-1}|0\rang = \]
\[\frac{\lang [I_{(1^k)}],\ \alpha_{-1}[I_{(1^{k-1})}]\rang 
\lang [I_{(1^{k-1})}],\ \Gamma_+(m,z)^*|0\rang}
{\langle [I_{(1^{k-1})}],\ [I_{(1^{k-1})}]\rangle} = 
kz^{-k}\prod_{j=0}^{k-1}(m-jt_2),\]
\[\lang [I_{(l)}],\ \Gamma_-(m+t_1+t_2,z)^*\alpha_{-1}|0\rang = 
lz^l\prod_{j=0}^{l-1}(m+t_1+t_2+jt_1).\]
Adding up each term yields
\[\langle \bW(\cL,z)[I_\mu],\ [I_\lambda]\rangle = 
z^{l-k}\prod_{j=0}^{k-1}(m-jt_2)\prod_{j=0}^{l-1}(m+t_1+t_2+jt_1)+\]
\[t_1t_2klz^{l-k}\prod_{j=0}^{k-2}(m-jt_2)\prod_{j=0}^{l-2}(m+t_1+t_2+jt_1) =\]
%\[z^{l-k}(m+t_2)(m+lt_1-(k-1)t_2)\prod_{j=0}^{k-2}(m-jt_2)\prod_{j=0}^{l-2}(m+t_1+t_2+jt_1)\]
\[z^{l-k}(m+lt_1-(k-1)t_2)\prod_{j=0}^{k-2}(m-jt_2)\prod_{j=0}^{l-1}(m+t_2+jt_1)\]
which agrees with (\ref{combtheorem}).

\end{proof}

\subsection{Proof of the main theorem}
\label{proof}
Before proving theorem \ref{mainT}, we prove two lemmas that reduce our
considerations to setup (\ref{csetup}).
\begin{Lemma}
\label{lehn}
If theorem \ref{mainT} holds in the case where $S$ has a torus action with 
isolated fixed points, it holds for all smooth quasiprojective surfaces with
compact fixed loci (and therefore all smooth, projective varieties).
\end{Lemma}
\begin{proof}
We prove the equality of $W(\cL)$ and $\Gamma(\cL)$ in the general case
by looking at a suitable class of matrix elements of each.
As before, let $\cI_n$ on $S^{[n]} \times S$ denote the canonical ideal
sheaf provided by the functor of points. Given $c \in H^j(S)$, let
\[
\sigma_i(c) = \int_S \ch_{i+2}(\cI) \cup c \quad \in H^{2i+j}(S^{[n]}) \,.
\]
It is proved, for instance in \cite{LQW} that polynomials in
these classes span $H^*(S^{[n]})$, as a vector space. Therefore, it
suffices to check the equality on the matrix elements
\begin{equation}
\label{matrix_coeff}
\left(\bW(\cL) \prod_{i=1}^{N_\eta} \sigma_{p_i}(\eta_i), \prod_{j=1}^{N_\xi} \sigma_{q_j} (\xi_j)\right),
\end{equation}
where $p_i$, $q_i$ are some integers. The order in the product here matters up to a
change in sign, since the cup product is supercommutative.

By Grothendieck-Riemann-Roch, the element $\bW(\cL) \in H^*(S^{[k]}\times S^{[k]})$
is the integral over $S$ of some polynomial in the classes
\[\left\{\ch_i(\cI_k),\ch_i(\cI_l),\ch_i(\cL),\Td(S)\right\}\]
on $S^{[k]}\times S^{[l]} \times S$. Then the matrix element (\ref{matrix_coeff}) is
the integral over
\[S^{[k]}\times S^{[l]} \times S \times S_1 \cdots \times S_{N_\eta+N_\xi}\]
of a polynomial $f$ of
\[\left\{pr^*_{S^{[k]}\times S}\ch_p(\cI_k),pr^*_{S^{[l]}\times S}\ch_p(\cI_l),
pr^*_{S}\ch_p(\cL),pr^*_{S}\Td(S),pr_{S_i}^*\eta_i,pr_{S_{N_\eta+i}}^*\xi_j \right\},\]
with each $S_i$ a different copy of $S$.

The method of Ellings\"{u}d, G\"{o}ttsche, and Lehn (\cite{EGL}, Proposition 3.1), shows that this 
becomes the integral of another polynomial $\tilde{f}$ depending only on $f$ on
\[S\times S_1 \times \cdots \times S_{k+l+N_\eta+N_\xi}\]
in the variables
\begin{equation}
\label{canonclasses}
\left\{\ch_p(pr^*_{ij}\cO_\Delta), \ch_p(pr^*_{i}T_S),\eta_i,\xi_j).\right\}
\end{equation}
This statement is 
also clearly true when we replace $\bW$ with $\Gamma$ in (\ref{matrix_coeff}).
Therefore, to check the equality $\bW(\cL) = \Gamma(\cL)$, it suffices to 
verify the equality of the polynomials in either case.

Checking the identity for every toric surface is enough to 
determine the coefficients. If there were two polynomials $\tilde{f}_1$,
$\tilde{f}_2$, then
\[\int_{S\times\cdots \times S} \tilde{f}_1-\tilde{f}_2 = 0\]
for all toric surfaces $S$, and there would be a universal relation among Chern
classes of the sheaves in (\ref{canonclasses}). Such a relation does not exist.

\end{proof}

Let $T$ act on $S$ with isolated fixed points.
\begin{Lemma}
\label{decomp}
In this decomposition,
\begin{enumerate}
\item $\bW_S(\cL) = \bigotimes_{s\in S^T} \bW_{T_s S}\left(\cL\big|_{\C^2_s}\right)$
\label{Wdecomp}
\item $\Gamma_S(\cL) = \bigotimes_{s\in S^T} \Gamma_{T_s S}\left(\cL\big|_{\C^2_s}\right)$
\label {Gdecomp}
\end{enumerate}

\end{Lemma}

\begin{proof}
(\ref{Wdecomp}) follows from the naturality property of characteristic classes.
(\ref{Gdecomp}) follows from (\ref{alphadecomp}). 
\end{proof}
From these two lemmas, it is sufficient
to prove the theorem in the setup (\ref{csetup}).
However, in the proof, we still appeal to other
surfaces. $\bW(\cL)$ is defined in terms of a localization integral, which is
equal to an honest integral in the case where $S$ is projective. This fact
by itself imposes relations on $\bW(\cL)$ which lead to an induction step.

\begin{proof}

By lemmas \ref{lehn} and \ref{decomp}, it suffices to prove the theorem for (\ref{csetup}).
However, we will make use of relations on $\bW(\cL)$ coming from
the fact that equivariant integrals of top degree classes on compact spaces 
are just numbers that correspond to normal integrals. Consider the following setup:
\[S=\Pp \times \Pp\,,\quad T=\Cx\times \Cx\times \Cx\,,\quad \cL = \C,\]
where the group action on $\Tot(\cL) \cong \Pp \times \Pp \times \C$ is
\[(z_1,z_2,u)\cdot (x_1,x_2;s) = (z_1x_1,z_2x_2;us).\]

On $\Pp \times \Pp$ we have equivariant cycles
\[L_1^0 = \{0\} \times \Pp, \quad L_1^\infty = \{\infty\} \times \Pp, \quad L_2^0 = \{0\} \times \Pp,
 \quad L_2^\infty = \{\infty\} \times \Pp. \]
Also let $U_{ab} \equiv \C^2$ be the $T$-equivariant chart of $\Pp \times \Pp$ by
containing the point $a,b \in \{0,\infty\}$
Given partitions $\mu$ and $\nu$ such that $|\mu|=k$ and $|\nu|=l$, define
\begin{equation}
w_{\mu\nu} = \left(\bW_S(\cL) \prod \alpha_{-\mu_i}(L_1)  \lv0 \rang, \prod \alpha_{-\nu_i}(L_2)  \lv0 \rang\right) \in \Z\,,
\label{wmn}
\end{equation}
\begin{equation}
w^{[ab]}_{\mu\nu} = \left(\bW_{U_{ab}}(\cL) \prod \alpha_{-\mu_i}(L_1^a)  \lv 0\rang, 
\prod \alpha_{-\nu_i}(L_2^b)  \lv 0\rang\right) \in \C(t_1,t_2,m),
\label{wabmn}
\end{equation}
and similar quantities $g_{\mu\nu}, g^{[ab]}_{\mu\nu}$ replacing $\bW$ with $\Gamma$.
Notice that $w_{\mu\nu}$ may be computed either as a non-equivariant integral, or
using localization. To compute it using localization, we must make a choice of
$L_\epsilon^0$ versus $L_\epsilon^\infty$ for each occurence of
$L_\epsilon$ in the expression (\ref{wmn}), though the answer is
independent of this choice. This corresponds to a decomposition 
$\mu = \mu^{[0]}\sqcup \mu^{[\infty]}$, $\nu = \nu^{[0]}\sqcup \nu^{[\infty]}$.
Using all three parts of lemma \ref{decomp} we arrive at an expression relating
$w_{\mu\nu}$ and $w_{\mu\nu}^{[ab]}$,
\begin{equation}
\label{indrelation}
 w_{\mu\nu} = \sum_{\mu^{[00]},\mu^{[0\infty]},\mu^{[\infty0]},\mu^{[\infty\infty]}}\quad
\sum_{\nu^{[00]},\nu^{[0\infty]},\nu^{[\infty0]},\nu^{[\infty\infty]}}
\prod_{a,b} w^{[ab]}_{\mu^{[ab]}\nu^{[ab]}}
\end{equation}
where the sum is over terms such that
$\mu^{[a0]} \sqcup \mu^{[a\infty]} = \mu^{[a]}$, 
$\nu^{[0b]} \sqcup \nu^{[\infty b]} = \nu^{[b]}$.
The same relation holds replacing $w$ with $g$.
We can now prove by induction on $\ell(\mu)$, $\ell(\nu)$ that $w^{[ab]}_{\mu\nu} = g^{[ab]}_{\mu\nu}$.
As noted above, the case $w^{[00]}_{\mu\nu} = g^{[00]}_{\mu\nu}$ proves the theorem. 

Assume that
$\ell(\mu) > 1$ or $\ell(\nu) > 1$, and that $w^{[ab]}_{\mu'\nu'} = g^{[ab]}_{\mu'\nu'}$ for
$(\ell(\mu'),\ell(\nu')) < (\ell(\mu),\ell(\nu))$. Choosing $\mu^{[0]} = \mu$, $\mu^{[\infty]} = \emptyset$, 
$\nu^{[0]} = \nu$, $\nu^{[\infty]} = \emptyset$  and 
solving (\ref{indrelation}) for $w^{[ab]}_{\mu\nu}$ yields a function 
of $w_{\mu\nu}$, $w^{[**]}_{\mu'\emptyset}$,
$w^{[**]}_{\emptyset\nu'}$, and $w^{[**]}_{\mu'\nu'}$ with 
$(\ell(\mu'),\ell(\nu')) < (\ell(\mu), \ell(\nu))$. On the other hand, any other choice
of $\mu^{[a]}$, $\nu^{[a]}$ expresses $w_{\mu\nu}$ in the lower order terms
$w_{\mu'\nu'}$. Combining these relations, we arrive at an expression of
$w^{[ab]}_{\mu\nu}$ in lower order terms $w^{[**]}_{\mu'\nu'}$ which also
holds for $g$. By induction, we are done.

What remains are the base cases $\mu = \emptyset$, $\nu = \emptyset$, and
$(\ell(\mu),\ell(\nu)) = (1,1)$. As noted in the above remark, the first two are proved by
Lehn's theorem. As for the third, it is in fact enough to demonstrate the claim
between any collection of elements $(\xi_k, \eta_l) \in S^{[k]}\times S^{[l]}$ 
that have nonzero inner product with $\alpha_{-k}|0\rangle$ and $\alpha_{-l}|0\rangle$.
This is done in lemma \ref{striplemma}.

\end{proof}

\section{Correlation Functions and Quasimodularity}
\label{quasi}
In this section, we investigate the aforementioned integrals on the Hilbert
scheme. We remain true to our strategy, and compute these integrals as
traces in the Nakajima basis. We eventually specialize to setup
(\ref{idealsetup}) mentioned in the introduction. We then prove
that the correlation functions are \emph{quasimodular forms}, and show
how this can be used to compute them.

\subsection{The Partition function}

To begin, suppose that the cohomology class
$c$ that we are integrating against is $1$. We use the following lemma,
\begin{Lemma} \label{CR}
We have the following commutation relations between vertex operators:
\begin{enumerate}
\item $[\Gamma_+(\cL,x),\ \Gamma_+(\cL',y)] = [\Gamma_-(\cL,x),\ \Gamma_-(\cL',y)] = 0$
\item $\Gamma_+(\cL,x)\Gamma_-(\cL',y) = 
(1-\frac{y}{x})^{-\lang c_1(\cL),c_1(\cL')\rang}\ \Gamma_-(\cL',y)\Gamma_+(\cL,x)$
\end{enumerate}
\end{Lemma}
Some explanation is necessary. First, if $x,y$ are viewed as formal variables,
the expression $(1-\frac{y}{x})^a$ must be taken as a formal series in
$\frac{y}{x}$. If we assign values to $x$ and $y$, the operator $\Gamma_{\pm}(\cL)$
does not take values in $\cF$. However, if we are interested the trace
of this operator, we only need its matrix elements. Matrix elements make perfect sense
for numerical choices of the variables as long as $x > y$ in the lemma. The proof 
of the lemma is obvious from the Nakajima commutation relations.
\qed

We may now calculate one of our integrals. Let $\cF_n$ be the graded 
component $F\otimes_R H^*(S^{[n]})$ in $\cF$, and $d$ the operator
that multiplies $\cF_n$ by $n$. Then for $|q| < 1$,
\[\sum_n q^n \int_{S^{[n]}} e(T_mS^{[n]}) = \Tr q^d \Gamma_-(\cL-\cK,z)\Gamma_+(\cL,z) = \]
\[\Tr \Gamma_-(\cL-\cK,zq) q^d \Gamma_+(\cL,z) = \Tr q^d \Gamma_+(\cL,z)\Gamma_-(\cL-\cK,zq) = \]

\[(1-q)^{\lang c_1(\cL),c_1(\cK)-c_1(\cL)\rang}\ 
\Tr q^d \Gamma_-(\cL-\cK,zq)\Gamma_+(\cL,z).\]
The second equality comes from the obvious fact that $d$ measures the 
grading, and $z$ measures how much $\Gamma_-(\cL,z)$ increases the grading.
Repeating the process of cycling $\Gamma_-(\cL)$ around the trace leaves
\[\lim_{n\rightarrow \infty} (q;q)_n^{\lang c_1(\cL),c_1(\cK)-c_1(\cL)\rang}\ 
\Tr q^d \Gamma_-(\cL-\cK,zq^n)\Gamma_+(\cL,z) = \]
\[(q;q)_\infty^{\lang c_1(\cL),c_1(\cK)-c_1(\cL)\rang}\ 
\Tr q^d\ \Gamma_-(\cL-\cK,0) \Gamma_+(\cL,z) = \]
\[(q;q)_\infty^{\lang c_1(\cL),c_1(\cK)-c_1(\cL)\rang}\ 
\Tr q^d\ = (q;q)_\infty^{\lang c_1(\cL),c_1(\cK)-c_1(\cL)\rang-\chi(S)}, \]
with $(a;q)_n = (1-a)(1-aq)\cdots(1-aq^n)$, and $\chi(S)$ is the Euler characteristic of $S$.
\subsection{Correlation functions}
\label{partitionsection}
Next, we turn to more general correlation functions, in setup (\ref{idealsetup}).
This is equivalent to making the specialization $t_1=-t_2 = t$ in (\ref{csetup}). 
We see that the vertex operator depends only on $\mu=m/t$, so we may further
assume $t=1$. 

Under this specialization, the relationship between the fixed-point basis and the Nakajima
basis become much simpler: On the symmetric function side, when $t_1=-t_2=1$
the Jack polynomials specialize to Schur polynomials. On the vertex operator
side, the operator $\Gamma(\cL,z)$ simplifies because the canonical bundle now
becomes trivial equivariantly (hence the term $\cK$ disappears). Also,
the operator $c \cup \_$ on cohomology
admits a simple expression in the vertex operators. This also stems from the vanishing
of $\cK$.

The combinatorics of this situation can be described by that of the \emph{charge zero
infinite wedge representation}. The description of the infinite wedge
representation is standard (see, for instance Ka\c{c} \cite{K}), but 
we recall it here. Let $V$ be the vector space with basis
$v_n, n \in \mathbb{Z}$. The \emph{infinite wedge space} $\fock$ is the vector
space with basis
\[v_I = v_{i_1} \wedge v_{i_2} \wedge \cdots\]
such that $i_1 > i_2 > ...$, and such that for $n >> 0$, $i_n = n+m-1$. The integer
$m$ is called the \emph{charge} of $v_I$. The charge induces a grading on 
$\fock$, and the charge $m$ component is labeled $\fockm$.
Let $Q:\bigwedge^\infty_m V \rightarrow \bigwedge^\infty_{m+1} V$ be the isomorphism
\begin{equation}
Q: v_{i_1} \wedge v_{i_2} \wedge \cdots \mapsto v_{i_1+1} \wedge v_{i_2+1} \wedge \cdots
\end{equation}
Finally let $\alpha_0$ be the operator that multiplies $\bigwedge^\infty_m V$ by $m$.

Each component has a basis indexed by partitions:
\[v_\mu = v_{\mu_1+m}\wedge v_{\mu_2+m-1} \wedge \cdots \wedge v_{\mu_{i}-i+m+1} \wedge \cdots\]
$\fock$ has operators $\psi_n, \psi^*_n$ given by
\[\psi_i(v_{i_1}\wedge v_{i_2}\wedge \cdots)  =  \begin{cases} 0 & \mbox{if $i = i_j$} \\ 
(-1)^{j} v_{i_1}\wedge \cdots \wedge v_{i_{j}} \wedge v_i \wedge v_{i_{j+1}} \wedge \cdots
& v_{i_j} > v_i > v_{i_{j+1}}\end{cases}\]
\[\psi_i^*(v_{i_1}\wedge v_{i_2}\wedge \cdots)  =  \begin{cases} 0 & \mbox{if $i \neq i_j$} \\ 
(-1)^{j-1} v_{i_1}\wedge \cdots \wedge v_{i_{j-1}} \wedge v_{i_{j+1}} \wedge \cdots
& i = i_{j}\end{cases}\]
These operators generate a Clifford algebra with commutation relations
\begin{equation}
\label{fermrelations}
\psi_i \psi^*_j + \psi_j^*\psi_i = \delta_{ij},
\quad \psi_i\psi_j = -\psi_i\psi_j,\quad \psi_i^*\psi_j^* = -\psi_i^*\psi_j^*,
\end{equation}
over which $\bigwedge^\infty V$ forms a representation called the 
\emph{infinite wedge representation}. In fact, $\bigwedge^\infty V$
is generated by applying the elements $\psi_{n+1},\psi^*_{-n}, n \geq 0$
to the \emph{vacuum vectors} $|m\rangle = v_m\wedge v_{m-1} \wedge \cdots$.

We are now in a position to describe the combinatorics of $\cF_{\C^2}$ with the above torus action.
Let $\Phi:\cF \rightarrow \bigwedge_0^\infty V$ be the isomorphism sending
\begin{equation}
[I_\mu] \mapsto (-1)^{|\mu|}\left(\prod_{\Box \in \mu} h(\Box)\right) v_\mu,
\end{equation}
where $h(\Box) = a_\mu(\Box)+l_\mu(\Box)+1$, the hook length of $\Box$.
Recall that $\alpha_k = (-1)^{k+1} \alpha_{-k}^*$. Thanks to the hook length
normalization above, the inner product transports to 
$\langle v_\mu, v_\lambda \rangle = \delta_{\mu,\lambda}$.
\begin{Proposition}{(Boson-Fermion correspondence)}
\label{BF}
Under the isomorphism $\Phi$,
\begin{equation}
\alpha_k \mapsto \sum_n \psi_{n} \psi^*_{n+k}
\end{equation}
Furthermore, we can recover the operators $\psi_i$ from $\alpha_k$ by
\begin{eqnarray}
\psi(x) & = &\sum_n \psi_nx^{n} = x^{\alpha_0}Q\Gamma_-(x)\Gamma_+(x)^{-1} \\
\psi^*(x) & = &\sum_n \psi^*_n x^{-n} = Q^{-1}x^{-\alpha_0}\Gamma_-(x)^{-1}\Gamma_+(x)
\end{eqnarray}
where 
\[\Gamma_{\pm}(x) = \exp(\sum_{n\geq 0} \frac{x^{\mp n}}{n} \alpha_{\pm n}).\]
\end{Proposition}
\begin{proof}
Upon setting $t_1=-t_2=1$, the Jack polynomial $\Jt_\mu$ becomes the Schur polynomial:
\[\Jt_\mu|_{t_1=1,t_2=-1} = (-1)^{|\mu|}\left(\prod_{\Box \in \mu} h(\Box)\right) s_\mu.\]
The rest may be found in Ka\c{c} \cite{K}, theorem (14.10).
\end{proof}

\subsection{Correlation functions and the vertex operator}
As pointed out, This representation of $\cF$ has the advantage that the
operator $c \cup \_$ takes a simple form. We now make use of this fact, and
compute the desired integrals. Specifically, we consider
\begin{equation}
F(k_1,...,k_N;m,q) = \sum_n q^n\int_{S^{[n]}} \ch_{k_1}(\cO/\cI)\cdots \ch_{k_N}(\cO/\cI) e(T_mS^{[n]})
\end{equation}
This is an equivariant integral, and such take values in the dual of $\Lie(T)$.
We evaluate this quantity at $(t=1,m)$, yielding a function of $q,m$ for
each $(k_1,...,k_N)$.

We can give an explicit combinatorial expression for $F$ using the localization
formula. The vector bundle $\cO/\cI$ at a fixed point $I_\mu$ has torus character
$\sum_{(i,j)\in \mu} e^{(j-i)t}$. Therefore
\[\ch_k(i_\mu^*\ \cO/\cI) = [t^k] \sum_{(i,j)\in \mu} e^{(j-i)t} = \sum_{(i,j)\in\mu} \frac{(j-i)^k}{k!}\]
The weights of the $\Cx$ action on $TS^{[n]}$ are given by $\{h(\Box),-h(\Box)\}|_{\Box \in \mu}$,
so the integral may be evaluated by localization:
\begin{equation}
F(k_1,...,k_N;q) = \sum_\mu q^{|\mu|} \left(\prod_l \sum_{(i,j)\in\mu} \frac{(j-i)^{k_l}}{k_l!}\right)
\prod_{\Box \in \mu} \frac{h(\Box)^2-m^2}{h(\Box)^2}.
\label{Floc}
\end{equation}

We represent the operator of cup-product against a $\ch(\cO/\cI)$
in the infinite wedge picture. The trick is to write $\ch(H^0(\cO/I)) = \ch(H^0(\cO))-\ch(H^0(I))$, and
express $\ch(\cI) \cup \_$ in the vertex operators.
\begin{Lemma}
Let $f_k :\bigwedge^\infty_0 V \mapsto \bigwedge^\infty_0$ be the operator
\[f_k(v_\mu) = \sum_{(i,j)\in\mu} \frac{(j-i)^{k}}{k!}v_\mu.\]
Then
\begin{equation}
f_k = \frac{[z^k]}{(2\pi i)^k}[y^0]((1-x^{-1})^{-1}-\psi(xy)\psi^*(y))(1-x)^{-1}
\end{equation}
with $x = e^{2\pi i z}$.
\end{Lemma}
\begin{proof}
The operator $[y^0]\psi(xy)\psi^*(y)$ is the same as
\[v_\mu \mapsto \left(\sum_{i \geq 1} x^{\mu_i-i+1}\right)v_\mu\]
which converges for $|x| > 1$. In this range, we get
\[((1-x^{-1})^{-1}-[y^0]\psi(xy)\psi^*(y))(1-x)^{-1} = \]
\[(\sum_{1 \leq i \leq \ell(\mu)} x^{-i+1}-x^{\mu_i-i+1})(1-x)^{-1} =\]
\[\sum_{1 \leq i \leq \ell(\mu)} x^{-i+1}+...+x^{\mu_i-i} =
\sum_{\Box = (i,j) \in \mu} x^{j-i}.\]
\end{proof}

Thus, it suffices to consider the auxilliary functions
\[G(x_1,...,x_N;y_1,...,y_N;q,m) = \]
\begin{equation}
\Tr_{\bigwedge_0^\infty}\ q^d\psi(x_1y_1)\psi^*(y_1)\cdots\psi(x_Ny_N)\psi^*(y_N)\ W(m,1).
\end{equation}
As pointed out before, for fixed numerical values of $x_i,y_i$, this is technically not 
an operator on $\bigwedge^\infty_0$, which is defined as a direct sum. However,
the matrix elements of this operator make perfect sense as long as
\begin{equation} \label{order}
|x_1y_1| > |y_1| > |x_2y_2| > |y_2| > ... > |x_Ny_N| > |y_N| > 1.
\end{equation}
Furthermore, the trace converges as long as $q$ is sufficiently small.
In this case, it is enough that $|x_1y_1| < |1/q|$. We now complete the 
final step of strategy \ref{strat}, and compute $G$ in this range.

Using the boson-fermion correspondence, theorem \ref{mainT}, and
the method of subsection \ref{partitionsection},
\[G(x_1,...,x_N;y_1,...,y_N;m,q) = \]

\[\Tr\ q^d \Gamma_-(x_1y_1)\Gamma_+(x_1y_1)^{-1}\Gamma_-(y_1)^{-1}\Gamma_+(y_1) \cdots\]
\[\Gamma_-(x_Ny_N)\Gamma_+(x_Ny_N)^{-1}\Gamma_-(y_N)^{-1}\Gamma_+(y_N) \Gamma_-(1)^m \Gamma_+(1)^{-m} =\]

\[\Tr\ \left[\Gamma_-(x_1y_1q)\right]\ q^d\ \Gamma_+(x_1y_1)^{-1}\Gamma_-(y_1)^{-1}\Gamma_+(y_1) \cdots\]
\[\Gamma_-(x_Ny_N)\Gamma_+(x_Ny_N)^{-1}\Gamma_-(y_N)^{-1}\Gamma_+(y_N) \Gamma_-(1)^m \Gamma_+(1)^{-m} =\]

\[\Tr\ q^d\ \Gamma_+(x_1y_1)^{-1}\Gamma_-(y_1)^{-1}\Gamma_+(y_1) \cdots\]
\[\Gamma_-(x_Ny_N)\Gamma_+(x_Ny_N)^{-1}\Gamma_-(y_N)^{-1}\Gamma_+(y_N) \Gamma_-(1)^m \Gamma_+(1)^{-m} 
\left[\Gamma_-(x_1y_1q)\right]=\]

\[\prod_{j=1}^N\left( 1-\frac{x_1y_1}{x_jy_j}q\right) \prod_{j=1}^N\left( 1-\frac{x_1y_1}{y_j}q\right)^{-1}
\prod_{j=1}^N\left( 1-x_1y_1q\right)^{m}\]
\[\Tr\ q^d\ \left[\Gamma_-(x_1y_1q)\right] \Gamma_+(x_1y_1)^{-1}\Gamma_-(y_1)^{-1}\Gamma_+(y_1) \cdots\]
\[\Gamma_-(x_Ny_N)\Gamma_+(x_Ny_N)^{-1}\Gamma_-(y_N)^{-1}\Gamma_+(y_N) \Gamma_-(1)^m \Gamma_+(1)^{-m} =\]
\[\cdots\]
\[(q;q)_\infty (x_1q;q)_\infty^{-1} \prod_{j=2}^N \left(\frac{x_1y_1}{x_jy_j}q;q\right)_\infty 
\prod_{j=2}^N \left(\frac{x_1y_1}{y_j}q;q\right)_\infty^{-1} \prod_{j=1}^N (x_1y_1q;q)_\infty^{m}\]
\[\Tr\ q^d\ \left[1 \right] \Gamma_+(x_1y_1)^{-1}\Gamma_-(y_1)^{-1}\Gamma_+(y_1) \cdots\]
\[\Gamma_-(x_Ny_N)\Gamma_+(x_Ny_N)^{-1}\Gamma_-(y_N)^{-1}\Gamma_+(y_N) \Gamma_-(1)^m \Gamma_+(1)^{-m}\]
Here we have repeatedly cycled around the trace as in subsection \ref{partitionsection}. The constraint
(\ref{order}) gives us the conditions we require for the commutation relations in lemma \ref{CR}.
We can repeat this process on $\Gamma_+(x_1y_1)^{-1}$, $\Gamma_-(y_1)^{-1}$, $\Gamma_+(y_1)$,
with the difference that the $\Gamma_+$ terms must move to the right, whereas the $\Gamma_-$ 
terms move left. Using the theta function 
$\theta(x;q) = (xq;q)_\infty(x^{-1};q)_\infty(q;q)_\infty^{-2}$, we arrive at
\[(q;q)_\infty^{2mN-1}\theta(x_1;q)^{-1} \prod_{j=2}^N \theta\left(\frac{x_1y_1}{x_jy_j};q\right) 
\prod_{j=2}^N \theta\left(\frac{x_1y_1}{y_j};q\right)^{-1} \prod_{j=1}^N \theta(x_1y_1;q)^{m}\]
\[\Tr\ q^d\ \left[1 \right] \Gamma_-(y_1)^{-1}\Gamma_+(y_1) \cdots\]
\[\Gamma_-(x_Ny_N)\Gamma_+(x_Ny_N)^{-1}\Gamma_-(y_N)^{-1}\Gamma_+(y_N) \Gamma_-(1)^m \Gamma_+(1)^{-m}=\]

\[(q;q)_\infty^{2mN}\theta(x_1;q)^{-1} \prod_{j=2}^N \theta\left(\frac{x_1y_1}{x_jy_j};q\right) 
\prod_{j=2}^N \theta\left(\frac{x_1y_1}{y_j};q\right)^{-1} \prod_{j=1}^N \theta(x_1y_1;q)^{m}\]
\[\prod_{j=2}^N \left(\frac{y_1}{x_jy_j}q;q\right)_\infty^{-1} 
\prod_{j=2}^N \left(\frac{y_1}{y_j}q;q\right)_\infty \prod_{j=1}^N (y_1q;q)_\infty^{-m}\]
\[\Tr\ q^d\ \left[1 \right]\Gamma_+(y_1) \cdots\]
\[\Gamma_-(x_Ny_N)\Gamma_+(x_Ny_N)^{-1}\Gamma_-(y_N)^{-1}\Gamma_+(y_N) \Gamma_-(1)^m \Gamma_+(1)^{-m}=\]

\[\prod_{j=2}^N \theta\left(\frac{x_1y_1}{x_jy_j};q\right) 
\prod_{j=2}^N \theta\left(\frac{x_1y_1}{y_j};q\right)^{-1} \prod_{j=1}^N \theta\left(x_1y_1;q\right)^{m}\]
\[\prod_{j=2}^N \theta\left(\frac{y_1}{x_jy_j};q\right)^{-1} 
\prod_{j=2}^N \theta\left(\frac{y_1}{y_j};q\right) \prod_{j=1}^N \theta(y_1;q)^{-m}\]
\[G(x_2,...,x_N;y_2,...,y_N;m,q) =\]

\begin{equation}
\label{Gtheta}
\left(\prod_i \frac{1}{\theta(x_i;q)}\right)
\left(\prod_{i < j} \frac{\theta(\frac{y_i}{y_j};q)\theta(\frac{x_iy_i}{x_jy_j};q)}
{\theta(\frac{y_i}{x_jy_j};q)\theta(\frac{x_iy_i}{y_j};q)}\right)
\left(\prod_i \frac{\theta(x_iy_i;q)}{\theta(y_i;q)}\right)^mZ(m,q).
\end{equation}
\newline
Here $Z(m,q) = G(\emptyset;\emptyset;m,q)$, which was calculated in the last subsection.
Taking $t=1$, we get
\begin{equation}
Z(m,q) = (q;q)_\infty^{m^2-1}.
\end{equation}
This turns into the following proposition about $F$, the original correlation functions:
\begin{Proposition}
\label{theta}
Let $x_k = e^{2\pi i z_k}$, $y_k = e^{2\pi i w_k}$, $\theta(z_k;q) = \theta(x_k;q)$. Then
\[F(k_1,...,k_n;m,q) = \int_{r_1i}^{r_1i+1}dw_1\cdots\int_{r_Ni}^{r_Ni+1}dw_N\]
\[\int_{\epsilon_1i}^{\epsilon_1i+1}\frac{dz_1}{z_1^{k_1+1}}\cdots
\int_{\epsilon_Ni}^{\epsilon_Ni+1}\frac{dz_N}{z_N^{k_N+1}} F(x_1,...,x_N;y_1,...,y_N;m,q),\]
where $-1 < r_1 < \cdots < r_N < 0$, $0<\epsilon_i < r_{i+1}-r_i$, the second $F$ is
given recursively by
\[F(x_1,...,x_N;y_1,...,y_N;m,q) = \frac{F(x_2,...,x_N;y_2,...,y_N;m,q)}{(1-x_1)(1-x_1^{-1})}-\]
\[\frac{G(x_1,...,x_N;y_1,...,y_N;m,q)}{1-x_1},\]
\[F(\emptyset;\emptyset;m,q) = Z(m,q) = (q;q)_\infty^{m^2-1},\]
and $G$ is the expression in (\ref{Gtheta}).
\end{Proposition}

\subsection{Quasimodularity of the correlation Functions}

If we know the coefficient of $z^k$ in $\theta(z;q) = \theta(x;q)$ for all $k$,
proposition \ref{theta} gives an explicit expression for $F(k_1,...,k_N;m,q)$.
However, a concise classification for these functions is preferable
to a complicated explicit expression. In this subsection, we give such a
classification that lands $F$ in a finite-dimensional vector space of functions of 
$q$, whose dimension depends on $k_j$. Specifically, we prove that $F$ is a 
\emph{quasimodular form} in the variable $q$.

Before stating the theorem, we classify the coefficients of 
$\theta(z;q)$. Let 
\[E_{2k} = -\frac{B_{2k}}{2k}+\sum_{n \geq 1}\left( \sum_{d|n}d^{k-1}\right)q^n,\]
the classical Eisenstein series, for $k\geq 1$. Here $B_{2k}$ are the Bernoulli numbers,
$B_2 = \frac{1}{6}$, $B_4 = -\frac{1}{30}$, $B_6 = \frac{1}{42}$, etc.
\begin{Lemma}
\[\theta(z;q) = a_1(q)z+a_3(q)z^3+...\]
where $a_k(q) \in \C[E_2,E_4,E_6]$. Furthermore, if $E_{2k}$ has
weight $2k$, then $a_k(q)$ has weight $2k-1$.
\end{Lemma}
\begin{remark} An element of $\C[E_2,E_4,E_6]$ is called a \emph{quasimodular form}
(for the full modular group $\Gamma$). Quasimodular forms were
introduced in \cite{KZ}. We do not develop their theory here, since we only use them
as a tool for classifying the coefficients of $\theta(z;q)$. They contain the set of Modular
forms, but also contain the element $E_2$, which has the property
\[E_2\left(\frac{a\tau+b}{c\tau+d}\right) =
(c\tau+d)^2E_2(\tau)-\frac{c(c\tau+d)}{4\pi i}.\]
\end{remark}
\begin{remark} The connection between quasimodular forms and the infinite wedge representation
is perhaps mysterious, but is proved by Bloch and Okounkov in \cite{BO}.
\end{remark}
\begin{proof}
This is lemma 6.2 of \cite{BO}.
\end{proof}
We now give an elegant classification of the correlation functions:
\begin{Theorem}
\label{quasiT}
As a function of $q$, $F(k_1,...,k_N;m,q)/Z(m,q)$ is a quasimodular form
of weight $\leq 2N+\sum_i k_i$. As a function of $m$, it is a polynomial of 
degree $2N+2\sum_i \lfloor \frac{k_i}{2} \rfloor$.
\end{Theorem}
\begin{proof}
If $R$ is a graded algebra, we get a grading on $R((z))$ by giving $z$ weight $-1$.
For instance, if $R=\C[E_2,E_4,E_6]$, then the power series of $\theta(z;q)$ 
about $z=0$ has weight $-1$ since the coefficient of $z^k$ has weight
$k-1$. On the other hand, $\theta(z,q)^{-1}$ has weight $1$.

Let
\[A = \prod_i \frac{1}{\theta(z_i,\tau)},\quad
B = \prod_{i < j} \frac{\theta(w_i-w_j;q)\theta(z_i+w_i-z_j-w_j;q)}
{\theta(w_i-z_j-w_j;q)\theta(z_i+w_i-w_j;q)},\]
\[C = \left(\prod_i \frac{\theta(z_i+w_i;q)}{\theta(w_i;q)}\right)^m,\quad
X' = \int_{ir_1}^{ir_1+1}dw_1\cdots\int_{ir_N}^{ir_N+1}dw_N X.\]
Here $X=B,C$, and $ABC$ is the expression in (\ref{Gtheta}). We start
by proving that the power series of $(BC)'$ in $z_j$ about $0$ is an
expression of weight $0$ in the algebra $R=\C[E_2,E_4,E_6]$.

In particular, we notice that $BC$ itself is a quotient $f/g$ of
holomorphic functions in $z, w$ of the same weight. 
We show that in the case $m \in \mathbb{Z}$, taking the integral
over $w_j$ reduces to a residue calculation, each one raising $\wt(f)-\wt(g)$
by $1$. To do this, notice that by the theta relation
\[\theta(z+\tau,\tau) = e^{-2\pi i(z-\tau)}\theta(z,\tau),\]
we find that $BC \mapsto e^{-2\pi i mz_1} BC$ under $w_1 \mapsto w_1+\tau$. Then
\[\int_{ir_1}^{ir_1+1}dw_1 BC= \frac{1}{1-e^{-2\pi i m z_1}}
\left(\int_{ir_1}^{ir_1+1}dw_1 BC-\int_{ir_1+\tau}^{ir_1+1+\tau}dw_1 BC\right) =\]
\[\frac{1}{1-e^{-2\pi i m z_1}} \int_{\Box} dw_1 BC,\]
where $\Box$ is the parallelogram with vertices $ir_1,ir_1+1,ir_1+1+\tau,ir_1+\tau$.
If $m$ is an integer, this becomes a simple residue calculation. It is easy to see
that taking the residue over $w_1$ yields a quotient $f/g$ in the remaining variables
such that $\wt(f) = \wt(g)-1$. 

By induction, it follows that
\[(BC)' = \left(\prod_{i=1}^N \frac{1}{1-e^{-2\pi i m z_i}}\right) D\]
where $D$ is a quotient $f/g$ such that $\wt(f)-\wt(g)=-N$. Since
the product has inhomogeneous weight $\leq N$ and $(BC)'$ is
holomorphic near $z_i = 0$, it must be inhomogeneous of weight $\leq 0$.
The remaining term $A$
has weight $N$, each term $(1-e^{2\pi iz_i})^{-1}$ appearing in proposition \ref{theta}
has weight $-1$, and differentiating raises
the weight by $1$. Therefore, $F(k_1,...,k_N;m,q)$ is quasimodular of weight $\leq 2N+k_1+...+k_N$
when $m$ is an integer.

It remains to show that $F$ is a polynomial in $m$. This also eliminates the
assumption that $m$ is integral, since any polynomial which is $0$ at every
integer is $0$. To do this, check that the coefficient
\[[m^n] \prod_i \left(\frac{\theta(z_i+w_i,\tau)}{\theta(w_i,\tau)}\right)^m=
\left(\sum_i \log\ \theta(z_i+w_i,\tau)-\log\ \theta(w_i,\tau)\right)^n\]
vanishes to order $n$ in $z_1,...,z_N$. In other words, taking fewer than
$n$ derivatives total in any of the variables $z_1,...,z_N$ and evaluating
at $z_j = 0$ gives zero identically in $y_j, \tau$. The expression $B$ is
holomorphic near $z_i = 0$, and
\[\frac{1}{(1-e^{2\pi i z_i})\theta(z_i,\tau)}\]
has a pole of order $2$ at $z_i=0$, and only even terms in the Laurent expansion.
This explains the funny expression for the degree as a polynomial in $m$.

\end{proof}

\begin{ex}
$F(1,3;q,m) = (2-\frac{5}{2}m^2+\frac{1}{2}m^4)q^2+(54-\frac{147}{2}m^2+21m^4-\frac{3}{2}m^6)q^3+...$
According to the theorem, $F(1,3,q,m)(q;q)_\infty^{1-m^2}$ is a polynomial of degree $6$ in 
$m$, and a quasimodular form of weight $\leq 8$ in $q$. A calculation using the low-order terms reveals
that
\[F(1,3;q,m)(q;q)_\infty^{1-m^2} = \left(\frac{2}{3}-\frac{1}{6}m^2-\frac{2}{3}m^4+\frac{1}{6}m^6\right) E_2^3 +\]
\[\left(-\frac{1}{6}-\frac{1}{24}m^2-\frac{1}{6}m^4-\frac{1}{24}m^6\right) E_2E_4 +\]
\[\left(-\frac{35}{3600}+\frac{35}{14400}m^2+\frac{35}{3600}m^4-\frac{35}{14400}m^6\right) E_6 +\]
\[\left(-16+28m^2-14m^4+2m^6\right) E_2^4 +\]
\[\left(-4+\frac{7}{3}m^2+\frac{7}{3}m^4-\frac{2}{3}m^6\right) E_2^2E_4 +\]
\[\left(\frac{50}{21}-\frac{25}{9}m^2+\frac{25}{72}m^4+\frac{25}{504}m^6\right) E_4^2 +\]
\[\left(\frac{13}{30}-\frac{77}{120}m^2+\frac{7}{3}m^4-\frac{1}{40}m^6\right) E_2E_6.\]
\end{ex}

\section{The Moduli of Sheaves}
\label{sheaves}

We start with some very brief motivation from gauge theory, then move on to study
the partition function for sheaves of rank $2$. This material is located in \cite{O}.

\subsection{Instantons and the moduli space}

In quantum gauge theory, one seeks to integrate over the space of gauge fields, i.e.
rank-$r$ vector bundles on a Riemannian 4-manifold $S$ together with
a connection, modulo the action of the group of gauge symmetries. The integrand
may be a function of the form $\exp(-\beta S[A])$, where $S$ is some energy functional
of a connection $A$.
If $\beta$ is large, one may assume that the minimizers of $S$ dominate,
and reduce to an integral over the minimizers of $S$. One is then
led to integrals over finite-dimensional smooth manifolds.

The minimizers are called \emph{instantons}. Specifically, let
\[F_A = d_AA = dA+\frac{1}{2}[A,A]\]
be the curvature of $A$, and
\[S[A] = -\int_M \Tr(F_A \wedge *F_A),\]
where $*$ is the Hodge star.
The minimizers here are solutions to the Yang-Mills (anti) self-duality
equation
\[F_A\pm*F_A = 0.\]
A rank-$r$ instanton is a $U(r)$ vector bundle $E$ on $M$, together
with a connection satisfying the self-duality or anti self-duality equation.

The solutions are grouped by a discrete invariant, \emph{charge},
\[c_2 = \frac{1}{8\pi^2}\int_{\R^4} \Tr F^2.\]
The letter $c_2$ is used because each it represents the 
second Chern class of the bundle $E$. This is even so in the
noncompact case $M=\R^4$, because instantons on $\R^4$
extend to instantons on $S^4$. The resulting vector bundle
on $S^4$ may not be trivial, and in fact $c_2$ is the second Chern
class of this bundle in $H^4(S^4) \cong \Z$. Solutions to the
self duality equation have $c_2<0$, while the anti self duality
solutions have $c_2>0$. One can switch between $c_2>0$ and $c_2<0$
by reversing the orientation of $\R^4$, so it is fair to 
assume instantons are anti-self-dual with $c_2\geq 0$.

In Donaldson theory, charge $n$ $U(r)$ instantons on $\R^4$
are identified with a moduli space
\[\cM(r,n) = \left\{\left(E,\Phi\right)\right\},\]
Where $E$ is a rank-$r$ vector bundle on $\Pt$ with $c_2(E) = n$,
\[\Phi:E|_{\Pp_\infty} \xrightarrow{\cong} \cO_{\Pp_\infty}\]
is a \emph{framing}, i.e. a particular trivialization of $E|_{\Pp_\infty}$,
and $\Pp_\infty$ is the curve $z=0$ in $\Pt$. This space has the advantage
of a complex structure. Notice that $c_1(E) = 0$, just by the existence of 
the framing.

However, it is still not compact. As a first step
towards compactification we embed $\cM(r,n)$ into the larger moduli space
\[\overline{\cM(r,n)} = \left\{(\mbox{rank-$r$ torsion-free sheaf $F$}, 
\Phi:E|_{\Pp_\infty} \xrightarrow{\cong} \cO_{\Pp_\infty})\right\}.\]
\begin{ex} Let $F$ be a rank $1$ torsion-free sheaf with a framing.
We have a canonical map $F \rightarrow F^{\vee \vee} = \cO$ to the double dual,
which must be an injection. This means $F$ must come from an ideal sheaf on the
complement of $\Pp_\infty$ which is $\C^2$. Thus, $\cM(1,n) = \left(\C^2\right)^{[n]}$, the Hilbert
scheme on $\C^2$.
\end{ex}
As this example shows, $\overline{\cM(r,n)}$ is still not compact.
In the next subsection, we will see that 
we must again turn to equivariant integration.

We now give a rough description of what
the integrand $e(T_m M)$ means from the point of view of gauge theory.
Suppose that we would like to integrate not just over the space of connections
on a bundle $E$ (or the above compactification), 
but over connections together with a section of $E^{\vee}\otimes E$. In physics, this 
corresponds to a \emph{matter field} interacting with the field $A$. 
For any fixed field, the literal space
of sections is badly behaved from a mathematical point of view. A related but better
behaved quantity is
\[\chi(F,F) = \Ext^0(F,F)-\Ext^1(F,F)+\Ext^2(F,F)\]
which is a virtual vector space that extends to sheaves $F$ in the larger
moduli space $\overline{\cM(r,n)}$.

The simplest thing we could integrate over this space of ``sections'' is the
element $1$ in cohomology. To make sense of such an integral, we regularize by
adding a scaling action by $u\in\Cx$, and consider an equivariant integral. Localization
the leads to the quantity
\[\int_{\overline{\cM(r,n)}} \frac{1}{e(u\cdot \chi(F,F))},\]
which becomes rigorous after adding a group action to $\overline{\cM(r,n)}$,
and regularizing the infinite-dimensional virtual bundle which is
$\chi(F,F)$ over each point $F$.

To salvage the integral, we write
\[\chi(F,F) = \chi(\cO^r,\cO^r)-\left(\chi(\cO^r,\cO^r)-\chi(F,F)\right).\]
As we note in the next section, the term on the right is isomorphic to the
tangent bundle to $\overline{\cM(r,n)}$. The term on the left is infinite, but 
uninteresting. Taking out it's contribution (called $Z_{pert}$),
the integral that remains is
\[\int_{\overline{\cM(r,n)}} \frac{1}{e(u\chi(F,F)-u\chi(\cO^r,\cO^r))}=
\int_{\overline{\cM(r,n)}} e(T_m\overline{\cM(r,n)}).\]
When $r=1$, this is the partition function from the last section.

\subsection{The Nekrasov partition function}
\label{neksection}
As for the Hilbert scheme, a torus action
$T \circlearrowleft \Pt$ induces a torus action on the moduli space. Unfortunately, this action does not have 
isolated or even compact fixed loci for $r > 1$. To resolve this problem, we follow
\cite{NO} and enlarge the group to include the action of the group $GL(r)$ on the choice of
framing by composition on the right. In fact, it is enough to include the action of the standard
torus $T^r \in GL(r)$.

We can now define the Nekrasov partition function. Consider the action
\[G = T\times T^r \times \Cx \circlearrowleft \overline{\cM(r,n)}\] 
where the second term acts as described, $\Cx$ acts trivially, and
the first $T$ action is induced from the action (\ref{idealsetup}) on $\C^2\subset \Pt$.
The partition function is defined as a product of a factor (called $Z_{pert}$) times the series,
\begin{equation}
Z_{inst}(a_1,...,a_r;t;m,q) = \sum_n q^{rn} \int_{\overline{\cM(r,n)}} e(T_m\overline{\cM(r,n)}).
\end{equation}
We describe the factor below.

The parameters $a_j$ here lie in $\Lie(T^r)$, and $m \in \Cx$ acts by scaling the
tangent bundle as before. To give a combinatorial expression using localization,
we must find the character of the tangent bundle to the fixed sheaves.

\begin{ex} Let $I,J$ be a pair of ideal sheaves such that $\cO/I$,
$\cO/J$ are zero dimensional and supported away from $\Pp_\infty$. Then 
$F=I\oplus J$ is a rank $2$ torsion-free sheaf on $\Pt$. The map $I,J \rightarrow \cO$
restricted to $\Pp_\infty$ gives an automatic framing. Furthermore, each $I,J$ has
a two-step resolution by vector bundles, which shows that $\ch_2(F) = c_2(F) = \dim_\C(\cO/I)+\dim_\C(\cO/J)$,
and $c_1(F) = \ch_1(F) = 0$.
\end{ex} 

\begin{Claim}
\label{sheavesfixed}
The fixed points of the action of $T \times T^r$ on $\overline{\cM(r,n)}$ are the
sheaves $I_{\mu^{[1]}} \oplus \cdots \oplus I_{\mu^{[r]}}$.
\end{Claim}

Again, as a representation we have
\begin{Claim}
as a virtual representation of $G$, the tangent space to the fixed sheaf
$F=I_{\mu^{[1]}}\oplus \cdots \oplus I_{\mu^{[r]}}$ is
\[T_F\overline{\cM(r,n)} = 
\chi_{\C^2}(\bigoplus_j e^{a_j}\cO,\bigoplus_j e^{a_j} \cO) - \chi_{\C^2}(F,F) = \]
\begin{equation}
\label{TF}
\bigoplus_{i,j} e^{a_i-a_j}\left(\chi_{\C^2}(\cO,\cO)-\chi_{\C^2}(I_{\mu^{[i]}},I_{\mu^{[j]}})\right).
\end{equation}
\end{Claim}
This too is in \cite{O}. Once again, we apply the localization theorem to get
\begin{equation}
Z_{inst} = \sum_{\mu^{[1]},...,\mu^{[r]}}q^{rn}
\prod_k\prod_{\Box \in \mu^{[k]}}
\frac{(m+a_i-a_j+t\hook(\Box))(m+a_i-a_j-t\hook(\Box))}
{(a_i-a_j+t\hook(\Box))(a_i-a_j-t\hook(\Box))}
\label{Z}
\end{equation}
with $n=|\mu^{[1]}|+\cdots+|\mu^{[r]}|$.

\subsection{The Dual partition function}

It seems natural here to define the vertex operator analogously to the operator $\bW(\cL,z)$
on the Hilbert scheme. While there are operators analogous to the Nakajima
operators on cohomology, (cf. Licata \cite{Lic} and Baranovsky \cite{Ba}), their commutation relations 
are those of the affine Lie algebra $\slc$, which are more complicated than those of
the Heisenberg algebra. At any rate, it is unclear what should replace
theorem \ref{mainT}.

Instead, we study the \emph{dual partition function}, which is a reduction of $Z$
to the rank $1$ case. Among other things, this has the advantage that we may
use the vertex operator $\bW(m)$ as it is. We see that in the analysis of
$Z^{\vee}$, the representation of $\slc$ on $\focko$ is critical. 
We have no geometric basis for introducing $\slc$, but its appearance is perhaps not
surprising in light of the $\slc$ action above. It seems likely that the original
partition function may be studied using a generalization of $\bW(m)$,
and the natural $\slc$ action.

Given any collection of partitions $\mu^{[1]},...,\mu^{[r]}$ as
in expression (\ref{Z}), and a collection of integers $b_1+\cdots+b_r=0$,
we define the \emph{blended partition}
\begin{equation}
\mu = \bl(b_1,...,b_r;\mu^{[1]},...,\mu^{[r]})
\end{equation}
as the unique partition such that.
\[\left\{\mu_j-j+1\right\} = \bigsqcup_k \left\{r(\mu^{[i]}_j-j+b_i)+i\right\}\]
Thus, we have a bijection between the set of partitions and the set of
$r$-tuples of partitions together with an $r$-tuple of integers summing to $0$.

Now assume $r=2$. We recover the functions $w_m(\mu^{[1]},\mu^{[2]})$ from $w_m(\mu)$.
Using the expression (\ref{TF}), we see that
\[\frac{w_m(\mu)}{w_m(\nu(b))} = 
w_m(\mu^{[1]},\mu^{[2]})|_{t=2t,a_1=2bt,a_2=(-2b-1)t},\]
where
\[\nu(b) = \bl(b,-b;\emptyset,\emptyset) = 
\begin{cases} (2b,2b-1,...,1,0) & b \geq 0 \\
(-2b-1,-2b-2,...,1,0) & b < 0
\end{cases}
\]
Also,
\[|\mu| = 2|\mu^{[1]}|+2|\mu^{[2]}|+2b^2+b.\]

We now define the dual partition function:
\[Z^{\vee}(\zeta;t;m,q) = \sum_{2b \in \Z} e^{\zeta \cdot b}  w_m(\nu(b)) q^{2b^2+b}Z_{inst}(2bt,(-2b-1)t;2t;m,q).\]
According to the relations above, and again assuming without loss of generality that $t=1$,
\begin{equation}
Z^{\vee}(\eta;m,q) = \sum_\mu q^{|\mu|}e^{2\zeta\cdot b(\mu)} w_m(\mu).
\end{equation}
This is a sort of Fourier transform of the original function $Z$, with 
weights $e^{\zeta \cdot b}$ keeping track of the shift $b$
of the partition $\mu$, often referred to as the \emph{charge} of $\mu$
(not the same the previous notion of charge).
The remainder of the paper focuses on an analog of theorem \ref{quasiT} for this function.

\subsection{Charge and the affine Lie algebra}

To employ the vertex operator, we must represent the operator
$h_0:\focko \rightarrow \focko$ given by
\[v_\mu \mapsto 2b v_\mu,\]
where $b$ is the charge of $\mu$ from the previous subsection. In fact, this
operator can be expressed in terms of a natural representation of the affine
lie algebra $\slc$ on $\focko$. The representation can be found in Ka\c{c} \cite{K},
but we describe it briefly here.

Let
\[e = \left(\begin{array}{cc} \quad &1 \\ & \quad \end{array}\right), \quad
h = \left(\begin{array}{cc} 1 & \quad \\ \quad  & -1\end{array}\right), \quad
f = \left(\begin{array}{cc} \quad  & \\ 1& \quad \end{array}\right),\]
be basis vectors of $sl_2\C$. We obtain an action of the
Lie algebra $\C[t,t^{-1}]\otimes sl_2\C$ on $V=\C\cdot\Z$ by
identifying $V$ with $\C[t,t^{-1}] \otimes \left(\C\cdot u_0 \oplus \C\cdot u_1\right)$ by
\[t^j \otimes u_i \leftrightarrow v_{-2j-i}\]
Associate $x_j = t^j \otimes x$ for $x=e,h,f$. Other than $h_0$, each operator $x_j$ 
induces an operator on $\focko$ defined by
\[x_j\cdot v_{i_1} \wedge v_{i_2} \wedge \cdots = 
\sum_k v_{i_1} \wedge \cdots \wedge \left(x_j\cdot v_{i_k}\right) \wedge \cdots\]
$h_0$, however, leads to an infinite sum. The solution is to subtract the infinite part
of the sum:
\[h_0 v_I = h_0 \cdot v_{i_1} \wedge v_{i_2} \wedge \cdots = \]
\[\left(|I \cap \Z^{even}_{>0}|-|I \cap \Z^{even}_{\leq 0}|-
|I \cap \Z^{odd}_{>0}|+|I \cap \Z^{odd}_{\leq 0}|\right)v_I\]
This gives a projective representation of $\C[t,t^{-1}]\otimes sl_2\C$, which extends to 
an honest representation of $\slc$ by sending $K \mapsto 1$. In this representation, we
see that
\[h_0\cdot v_{\mu} = 2b(\mu)v_{\mu},\]
where $b(\mu)$ is the charge as defined above. To understand $Z^{\vee}$
We must understand the interaction between $h_0$ and the vertex operator.

\subsection{The Principal vertex operator construction}

Strategy (\ref{strat}) has allowed us to compute correlations only for operators
on $\cF$ expressed in the Nakajima basis. To understand $Z$, we must find such
an expression for
\[h_0 : \focko \rightarrow \focko.\]
In fact, the entire action of $\slc$ admits such a description, known as 
the \emph {principal vertex operator construction} \cite{K}:
\begin{Proposition} (principal vertex operator construction for $\slc$).
\label{pvoc}
Let
\[\Gamma^{odd}(z) = \sum_n \Gamma^{odd}_nz^n = \Gamma_-^{odd}(z)^2\Gamma_+^{odd}(z)^{-2},\]
where
\[\Gamma_{\pm}^{odd}(z) = \exp\left(\sum_{n>0,\ n\ odd} \frac{z^{\mp n}}{n}\alpha_{\pm n}\right).\]
The representation $\focko$ of $\slc$ is given in the Nakajima basis by
\[2d+h_0 \mapsto \sum_{j\geq 0} \alpha_{-2j-1}\alpha_{2j+1},\quad 
K \mapsto 1,\quad e_j+f_{j+1} \mapsto \alpha_{2j+1},\quad\]
\begin{equation}
h_j \mapsto [z^{2j}]\frac{1}{2}\left(1-\Gamma^{odd}(z)\right),\quad
e_j-f_{j+1} \mapsto [z^{2j+1}]\frac{1}{2}\left(1-\Gamma^{odd}(z)\right)
\label{h}
\end{equation}
\end{Proposition}
This can be found in \cite{K}, chapter 14.
\begin{ex} \[h_0\cdot v_{[1]} = [z^0] \Gamma(z) v_{[1]} = 
\frac{1}{2}\left(1-(1+4\alpha_{-1}\alpha_{1})\right)v_{[1]}=\]
\[\frac{1-5}{2}v_{[1]} = -2v_{[1]}.\]
$-2v_{[1]}$ is indeed $h_0$ of $v_{[1]} = v_{1}\wedge v_{-1}\wedge v_{-2} \wedge \cdots$ 
since there is one additional positive odd index, and one missing nonpositive even index.
\end{ex}

\section{Vector-Valued Modularity}
Now that we have enough vertex operators, let us attempt an analogue of theorem \ref{quasiT}
for the Nekrasov partition function. We will see that this leads us to a
modularity condition for vector-valued operators.

\subsection{A Vertex operator calculation}

Let
\[G(k;m,q) = [a_1^0\cdots a_k^0] G(a_1,...,a_k;m,q),\]
\begin{equation}
G(a_1,...,a_k;m,q) = \Tr q^d \Gamma^{odd}(a_1)\cdots\Gamma^{odd}(a_k)\Gamma^m(1)\Gamma^{-m}(1),
\end{equation}
\[ |1/q| > |a_1| > \cdots > |a_k| > 1.\]
so that
\begin{equation}
Z_k^{\vee}(m,q) = \frac{1}{2^k}\left(G(0;m,q)-
\left(\begin{array}{c}k \\ 1 \end{array}\right)
G(1;m,q)+
\left(\begin{array}{c}k \\ 2 \end{array}\right)
G(2;m,q)-\cdots\right),
\end{equation}
where $Z^{\vee}_k(m,q) = [\zeta^k] Z^\vee(\zeta;m,q).$
\begin{Lemma}
\begin{equation}
G(a_1,...,a_k;m,q) = \frac{\left(\prod_{1\leq i < j \leq n}\theta(\frac{a_i}{a_j},q)^2\right)
\left(\prod_{1\leq i \leq n}\theta(a_i,q)^m\right)
(q;q)_\infty^{2k+m^2-1}}
{\left(\prod_{1\leq i < j \leq n}\theta(-\frac{a_i}{a_j},q)^2\right)
\left(\prod_{1\leq i \leq n}\theta(-a_i,q)^m\right)
(-q;q)_\infty^{2k}}
\label{Grtheta}
\end{equation}
\end{Lemma}
\begin{proof}
This follows from an almost identical approach to proposition \ref{theta}, 
and the commutation relations
\[\Gamma_+^{odd}(x)\Gamma_-^{odd}(y) = \Gamma_-^{odd}(y)\Gamma_+^{odd}(x)
\sqrt{\frac{1+\frac{y}{x}}{1-\frac{y}{x}}},\]
\[\Gamma_+^{odd}(x)\Gamma_-(y) = \Gamma_-(y)\Gamma_+^{odd}(x)
\sqrt{\frac{1+\frac{y}{x}}{1-\frac{y}{x}}},\]
\[\Gamma_+(x)\Gamma_-^{odd}(y) = \Gamma_-^{odd}(y)\Gamma_+(x)
\sqrt{\frac{1+\frac{y}{x}}{1-\frac{y}{x}}}.\]
\end{proof}

\subsection{An Auxilliary Riemann surface}
$G(k;m,q)$ is an integral of $G(a_1,...,a_k;m,q)$, which we now consider.
The difficulty in the integration comes from the
piece involving the term $m$, since it is not a meromorphic function 
when $m \notin \mathbb{Z}$. To study these integrals we must first find
the Riemann surface on which the function lives. We choose to use the
function
\[\theta_{11}(x,q) = \theta(xq^{\frac{1}{2}},q)(q;q)_\infty^3 = \]
\begin{equation}
\prod_{k\geq 1} (1+xq^{k-\frac{1}{2}})(1+x^{-1}q^{k-\frac{1}{2}})(1-q^k) = \sum_n x^nq^{\frac{n^2}{2}}
\end{equation}
in place of $\theta$ since it is more convenient for picturing the Riemann surface,
and is the standard normalization for presenting the \emph{functional equation} for
rank-$1$ theta functions. Making this replacement in $G$ has no effect on the integrals $G(k;m,q)$,
which are the quantities of interest.

Let
\[f(z,\tau) = \frac{\theta_{11}(z-\frac{1}{4},\tau)}{\theta_{11}(z-\frac{3}{4},\tau)}.\]
Then
\begin{equation}
\label{fintegral}
G(1;m,q) = \frac{(q;q)_\infty^{2+m^2-1}}{(-q;q)_\infty^2}\int_0^1 f(z,\tau)^m ,
\end{equation}
which turns out to be the critical step. The branch of $f(z,\tau)^m$ must
be chosen so that $f(0,\tau)^m = e^{\pi im}$ for (\ref{fintegral}) to hold.

$f(z,\tau)$ has a pole of order $1$ at the points 
$\frac{1}{4}+\frac{\tau}{2}+\Z+\tau\Z$, and a zero of order 1 at
$\frac{3}{4}+\frac{\tau}{2}+\Z+\tau\Z$. It follows that analytically continuing
$f(z,\tau)^m$ under any choice of branch in a counter-clockwise loop around
the point $\frac{1}{4}+\frac{\tau}{2}$ amounts to multiplication by $e^{-2\pi im}$.
One may also check that analytically continuing along the path $[0,\tau]$ multiplies
by $e^{\pi im}$.
\begin{figure}[htp]
\centering
\scalebox{0.44}{\includegraphics{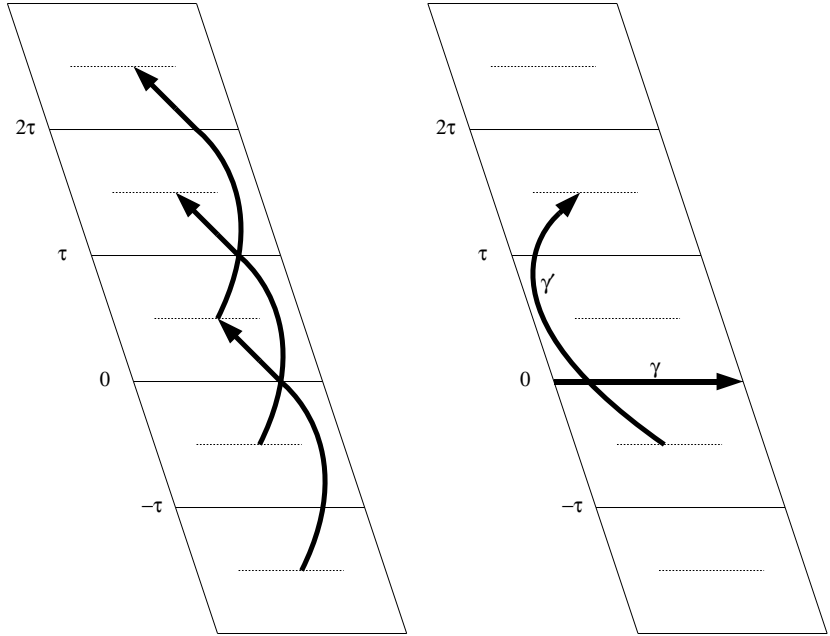}}
\caption{The Riemann surface $R$ is shown on the left. Here $z$ is glued to $z+1$.
On the left are two examples of homology cycles.}
\label{fig:riemann}
\end{figure}

$f(z,\tau)^m$ is a well-defined function on a Riemann surface $R$ defined as follows:
beginning with $\C$, glue $z$ to $z+1$. Make an incision between each pair of points
\[\left(\frac{1}{4}+\frac{\tau}{2}+a+b\tau,\frac{3}{4}+\frac{\tau}{2}+a+b\tau\right),\]
for pairs of integers $(a,b)$. Connect the region below the cut corresponding to
$(a,b)$ to the region above the cut corresponding to $(a,b-2)$. The rules in
the above paragraph show that $f(z,\tau)^m$ extends analytically across these gluings.
The surface and two important cycles are shown in figure (\ref{fig:riemann}).

\subsection{A Representation of a modular subgroup}

We describe how $f(z,\tau)^m$ behaves under changes of the coordinate $\tau$
in terms of an action of a subgroup $\Gamma \subset \slz$ on the homology cycles of the
Riemann surface $R$. Given any path $c$ in $\C^2$ avoiding the endpoints of the slits
$\frac{1}{4}+\frac{\tau}{2}+\frac{1}{2}\Z+\tau\Z$, we obtain a path in $R$.
If we identify any cycle $c+\tau$ with $e^{\pi im}c$, and make use of the
moves shown in figure (\ref{fig:move}), we obtain an element 
$[c] \in \C\cdot[\gamma] \oplus \C\cdot [\gamma']$. The identification
$[c+\tau]$ with $e^{\pi im}[c]$ is acceptable because we are
interested in evaluating
\[\int_cdz\ f(z,\tau)^m,\]
which is preserved since $f(z+\tau,\tau)^m = e^{\pi im}f(z,\tau)^m$. 
\begin{figure}[htp]
\centering
\scalebox{0.44}{\includegraphics{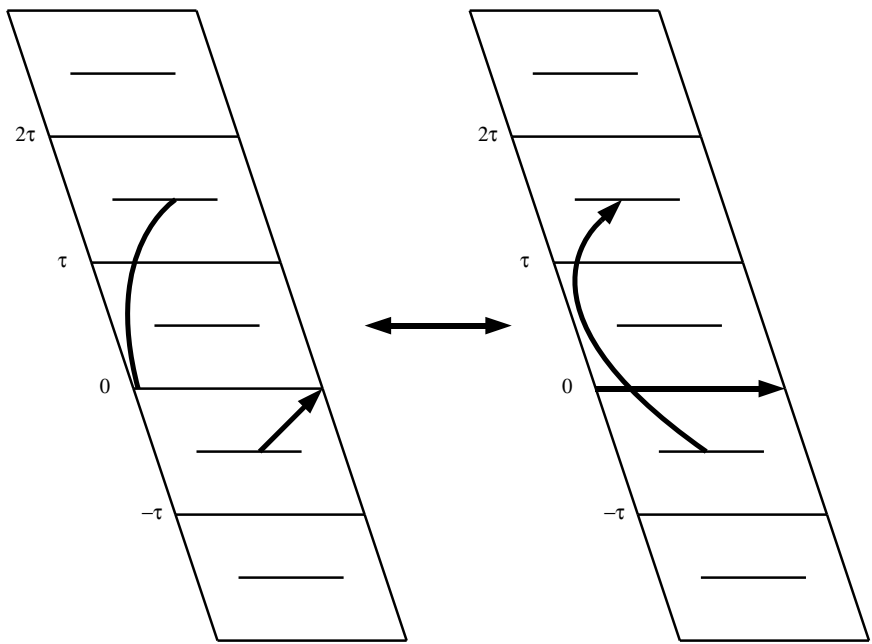}}
%\scalebox{0.44}{\includegraphics{}}
\caption{The cycles above are homotopy equivalent.}
\label{fig:move}
\end{figure}

Now let $\Gamma \subset \slz$ be the subgroup
\[\Gamma = \left(\begin{array}{cc} 4 & \\ & 1 \end{array}\right)\cdot \Gamma_{1,8}\cdot
\left(\begin{array}{cc} 4 & \\ & 1 \end{array}\right)^{-1} =\]
\[\left\{\left(\begin{array}{cc} a & b\\ c & d \end{array}\right): a,d \cong 1\ \mbox{(mod $8$)},
b \cong 0\ \mbox{(mod $4$)}, c\ \cong 0\ \mbox{(mod $2$)}\right\}.\]
Since $\Gamma$ preserves the points $\frac{1}{4}+\frac{\tau}{2}+\frac{1}{2}\Z+\tau\Z$
via its action on $\R^2 \cong \R\cdot \tau\oplus \R\cdot 1 \cong \C$, we obtain a representation 
\[\rho_m : \Gamma \rightarrow \C\cdot\gamma \oplus \C\cdot \gamma' \cong \C^2\]
by
\begin{equation}
\label{rho}
\rho_m(A)(a[\gamma]+b[\gamma']) = a[A\cdot \gamma]+b[A\cdot \gamma'].
\end{equation}
To write $[A\cdot\gamma^{(\epsilon)}]$ as a linear combination of $[\gamma]$, $[\gamma']$,
we use the equivalence $[\gamma^{(\epsilon)}+\tau] = e^{\pi im}[\gamma^{(\epsilon)}]$, and the homotopy equivalence in figure
\ref{fig:move}.

\begin{ex}
Let
\[A = \left(\begin{array}{cc} 1 & 4 \\ & 1 \end{array}\right).\]
Making use of the move (\ref{fig:move}), we see that
\[[A\cdot \gamma] \cong e^{-\pi im}[\gamma]+(1+e^{-\pi im})[\gamma'],\quad
[A\cdot \gamma'] \cong e^{\pi im}[\gamma'],\]
so that
\[\rho_m(A) = \left(\begin{array}{cc} e^{-\pi im} &  \\ 1+e^{-\pi im} & e^{\pi im} \end{array}\right).\]
\end{ex}

\subsection{Vector-valued modularity of $Z^\vee$ in rank $2$}
\begin{Lemma}
The vector-valued function $v(\tau)$ given by
\[v(\tau)_1 = \int_\gamma dz\ f(z,\tau)^m,\quad v(\tau)_2 = \int_{\gamma'} dz\ f(z,\tau)^m\]
satisfies the relation 
\begin{equation}
v(A\cdot \tau) = v\left(\frac{a\tau+b}{c\tau+d}\right)=\frac{1}{c\tau+d}\ \rho_m(A^t)\cdot v(\tau).
\end{equation}
\label{vv}
\end{Lemma}
\begin{proof}
$\theta_{11}(z,\tau)$ satisfies the functional relation
\begin{equation}
\label{thetar}
\theta_{11}\left(\frac{z}{c\tau+d},\frac{a\tau+b}{c\tau+d}\right) = \zeta_8^k\sqrt{c\tau+d}\ 
\exp(\frac{2\pi iz^2}{c\tau+d})\ \theta_{11}(z,\tau).
\end{equation}
which can be found, for instance, in \cite{MUM}. Then
\[f\left(\frac{z}{c\tau+d},\frac{a\tau+b}{c\tau+d}\right)=f(z,\tau).\]
So if $A^t \in \Gamma$,
\[\int_{[\gamma^{(\epsilon)}(A\cdot \tau)]}dz\ f(z,A\cdot\tau)^m =
\frac{1}{c\tau+d}\int_{(c\tau+d)[\gamma^{(\epsilon)}(A\cdot\tau)]}dz\ f\left(\frac{z}{c\tau+d},\frac{a\tau+b}{c\tau+d}\right)^m =\]
\[\frac{1}{c\tau+d}\int_{\rho_m(A^t) [\gamma^{(\epsilon)}(\tau)]}dz\ f(z,\tau)^m.\]
The result follows.

\end{proof}

\begin{Theorem}
 There exists a holomorphic vector-valued function 
$v_{k,m}$ on the upper-half-plane with values in  $\Sym^k \mathbb{C}^2$
satisfying
\[v_{k,m}(A\cdot \tau) = \frac{1}{(c\tau+d)^k}\, \Sym^k
\rho_m(A^t)\cdot v_{k,m}(\tau),\]
for $A\in\Gamma$, such that 
\[\frac{Z_k^\vee(m,\tau)}{(q;q)_\infty^{m^2-1}} = 
\frac{1}{2^k} \sum_{j=0}^k (-1)^{k-j}
\left(\begin{array}{c}k \\ j \end{array}\right)
\frac{(q;q)_\infty^{2j}}{(-q;q)_\infty^{2j}} v_{j,m}(\tau;m)_{1^k} =\]
\[\frac{1}{2^k} \sum_{j=0}^k (-1)^{k-j}
\left(\begin{array}{c}k \\ j \end{array}\right)
\frac{\eta(\tau)^{4j}}{\eta(2\tau)^{2j}} v_{j,m}(\tau;m)_{1^k}\]
Here $v_{k,m}(\tau;m)_{1^k}$ is the first coordinate of $v_{k,m}(\tau;m)$, and
$\eta$ is the Dedekind eta function, $\eta(q) = q^{\frac{1}{24}}\prod_{k\geq 1} (1-q^k)$
\label{vectorT}
\end{Theorem}

\begin{proof}
The function is
\[v_{k,m}(\tau;m)_{1^{k_1}2^{k_2}} = 
\int_{\gamma} dz_1\cdots \int_{\gamma} dz_{k_1} \int_{\gamma'} dz_{k_1+1}\cdots \int_{\gamma'} dz_{k_1+k_2}\]
\[\left(\prod_{i < j} f(z_i-z_j,\tau)^2\right)\left(\prod_i f(z_i,\tau)^m\right).\]
The proof is the same as for lemma \ref{vv}, and uses the fact that
the integrand on the left is doubly periodic in each $z_i$.
\end{proof}

\subsection{An Example}

Let us apply the theorem to compute $Z^{\vee}$ when $m$ is odd. 
The situation then greatly simplifies, since one can easily check that
$\rho_m(A)$ is diagonal for $A \in G$. Therefore
\[v_{k,m}(A\cdot \tau)_{1^k} = \frac{-1}{(c\tau+d)^k} v_{k,m}(\tau)_{1^k}\]
for $A\in G$. Since $v_{k,m}$ is periodic in $\tau$, we may
also drop the condition $b \cong 2\ \mbox{(mod $4$)}$, and extend the group
to all of $\Gamma_1(4)$. Since $\eta(2\tau)^2\eta(\tau)^{-4}$ satisfies
the same modularity condition, $Z_k^\vee(m,q)/(q;q)_\infty^{m^2-1}$ is modular of weight $0$.
We would like it to be a modular form, but it may have poles at the cusps of $X_1(4)$.
Multiplying by suitable power of the $\eta$ function takes care of the problem.

For instance, we see that
\[Z_2^\vee(3,q) = -16q+128q^5-320q^9+1120q^{17}-1024q^{21}+...\]
The function
\[\frac{Z_2^\vee(3,q)}{(q;q)_\infty^8}\eta(\tau)^4\eta(2\tau)^2\eta(4\tau)^4\]
does not have poles, is periodic in $\tau$, and is therefore a modular
form of degree $5$ for $\Gamma_1(4)$.
A quick calculation in \texttt{SAGE} (c.f. \cite{S}) shows that the vector space of such functions is
$3$-dimensional, spanned by
\begin{eqnarray*}
E_1(q) & = & q - 4q^2 + 16q^4 - 14q^5 - 64q^8 + 81q^9 +...\\
E_2(q) & = & 1 - 12q^2 - 128q^3 - 204q^4 - 1088q^6  - ...  \\
E_3(q) & = & q + 16q^2 + 80q^3 + 256q^4 + 626q^5 + 1280q^6 +...
\end{eqnarray*}
Looking at the low order coefficients, we see that it must be 
$(4/5)E_1(q)-(4/5)E_3(q)$.


\begin{thebibliography}{99}

\bibitem{AB} M.~Atiyah, R.~Bott,
\emph{The Moment map and equivariant cohomology},
Topology  23  (1984),  no. 1, 1 - 28.

\bibitem{Ba} V.~Baranovsky,
\emph{Moduli of sheaves on surfaces and action of the oscillator algebra},
 J. Differential Geom.  55  (2000),  no. 2, 193 - 227.

\bibitem{BO} S.~Bloch, A.~Okounkov,
\emph{The character of the infinite wedge representation},
Adv. Math.  149  (2000),  no. 1, 1 - 60.

\bibitem{Bo} A.~Borodin,
\emph{Periodic Schur process and cylindric partitions},\newline
\texttt{arXiv:math/0601019}

\bibitem{CO} E.~Carlsson, A.~Okounkov,
\emph{Exts and vertex operators},
\texttt{arXiv:0801.2565v1}

\bibitem{C} J.~Cheah,
\emph{Cellular decompositions for nested Hilbert schemes of points},
Pacific J. Math.  183  (1998),  no. 1, 39 - 90.

\bibitem{EGL} G.~Ellingsrud, L.~G\"ottsche, M.~Lehn,
\emph{On the cobordism class of the Hilbert scheme of a surface}
Journal of Algebraic Geometry, \textbf{10} (2001), 81 - 100. 

\bibitem{Goe}
L.~G\"ottsche,
\emph{Hilbert schemes of points on surfaces}, 
ICM Proceedings, Vol.\ II (Beijing, 2002), 483--494. 

\bibitem{Goe2} L.G\"ottsche,
\emph{The Betti numbers of the Hilbert scheme of points on a smooth projective surface},
Math. Ann.  286  (1990),  no. 1-3, 193 - 207. 

\bibitem{Groj} I.~Grojnowski,
\emph{Instantons and affine algebras I: the Hilbert scheme
and vertex operators}, 
Math.\ Res.\ Lett.\ {\bf 3} (1996), 275--291.

\bibitem{H} M.~Haiman,
\emph{Combinatorics, symmetric functions, and Hilbert schemes},
Current developments in mathematics, 2002,  39 - 111, Int. Press, Somerville, MA, 2003.

\bibitem{HL} D.~Huybrechts, M.~Lehn,
\emph{The geometry of moduli spaces of sheaves}, 
Aspects of Mathematics, E31. Friedr. Vieweg $\&$ Sohn, Braunschweig, 1997.

\bibitem{K} V.~Ka\c{c},
\emph{Infinite dimensional Lie algebras, third edition},
Cambridge University Press, 1990.

\bibitem{KZ} 
M.~Kaneko and D.~Zagier, 
\emph{A generalized Jacobi theta function and quasimodular forms},
The moduli space of curves, Progress in Mathematics,
\textbf{129}, Birkh\"auser, 1995.

\bibitem{L} M.~Lehn,
\emph{Geometry of Hilbert schemes}, 
CRM Proceedings and Lecture Notes, Volume 38, 2004, 1 - 30.

\bibitem{L2} M.~Lehn,
\emph{Chern classes of tautological bundles on Hilbert schemes of points on surfaces},
Invent. Math.  136  (1999),  no. 1, 157 - 207.

\bibitem{Lic} A. ~Licata,
\emph{Framed rank $r$ torsion-free sheaves on $\Pt$ and representations of the affine Lie algebra $\widehat{gl(r)}$},
\texttt{arXiv:math/0607690}

\bibitem{LQW} 
W.~Li, Z.~Qin, W.~Wang,
\emph{Vertex algebras and the cohomology ring structure of Hilbert schemes of points on surfaces}, 
Math.\ Ann.\ \textbf{324} (2002), 105 - 133. 

\bibitem{LQW_Jack} 
W.~Li, Z.~Qin, W.~Wang,
\emph{The cohomology rings of Hilbert schemes via Jack polynomials}, 
CRM Proceedings and Lecture Notes, vol.\ \textbf{38} (2004), 249--258. 

\bibitem{Mac}
I.~Macdonald, 
\emph{Symmetric functions and Hall polynomials}, 
The Clarendon Press, Oxford University Press, New York, 1995.

\bibitem{MO}
D.~Maulik and A.~Okounkov,
\emph{Nested Hilbert schemes and symmetric functions}, 
in preparation. 

\bibitem{MUM} 
D.~Mumford,
\emph{Tata Lectures on Theta},
Progress in Mathematics, \textbf{28}, Birkh\"auser, 1983.

\bibitem{Nak1} H.~Nakajima,
\emph{Heisenberg algebra and Hilbert schemes of points on projective surfaces},
Ann.\ of Math.\ (2) \textbf{145} (1997), no.~2, 379--388. 

\bibitem{Nak4} H.~Nakajima, 
\emph{Instanton counting on blowup. I. 4-dimensional pure gauge theory},
Invent. Math.  162  (2005),  no. 2, 313--355.

\bibitem{Nak3} H.~Nakajima,
\emph{Jack polynomials and Hilbert schemes of points on surfaces},
\texttt{arXiv:alg-geom/9610021}

\bibitem{Nak2} H.~Nakajima, \emph{Lectures on Hilbert schemes of points on surfaces},
AMS, Providence, RI, 1999.

\bibitem{NO}
N.~Nekrasov and A.~Okounkov,  
\emph{Seiberg-Witten Theory and Random Partitions}, 
In \emph{The Unity of Mathematics}
(ed. by P. Etingof, V. Retakh, I. M. Singer)
Progress in Mathematics, Vol.\ 244,
Birkh\"auser,  
2006, \texttt{hep-th/0306238}. 

\bibitem{O}
A.~Okounkov,  
\emph{Random Partitions and Instanton Counting}, 

International Congress of Mathematicians. Vol. III,  687 - 711, Eur. Math. Soc., Z\"{u}rich, 2006.

%\bibitem{O} A.~Okounkov,
%\emph{Random Partitions and Instanton Counting},
%International Congress of Mathematicians. Vol. III,  687 - 711, Eur. Math. Soc., Z�rich, 2006

\bibitem{OP} 
A.~Okounkov and R.~Pandharipande, 
\emph{Quantum cohomology of the Hilbert scheme of points in the plane},
\texttt{arXiv:math/0411210}. 

\bibitem{PS} 
A.~Pressley, G.~Segal, 
\emph{Loop Groups}, 
Clarendon Press, Oxford, 1986.

\bibitem{S}
W.~Stein,
\emph{Modular Forms, a Computational Approach},
American Mathematical Society, 2007.

\bibitem{Vass}
E.~Vasserot, 
\emph{Sur l'anneau de cohomologie du schema de Hilbert de $\mathbf{C}\sp 2$},
C.~R.\ Acad.\ Sci.\ Paris S\'er. I Math.\ \textbf{332} (2001), no.~1, 7 - 12.

\bibitem{ZHU} 
Y.~Zhu, 
\emph{Modular invariance of characters of vertex operator algebras}, 
J.~AMS, \textbf{9}
(1996), 237 - 302.

\end{thebibliography}
\end{document}